\documentclass[12pt,reqno]{amsart}
\usepackage{amsmath,amsthm,amssymb,mathrsfs,stmaryrd,color}
\usepackage[margin=1.3in]{geometry}
\usepackage[all]{xy}
\usepackage{url}
\usepackage{geometry}
\usepackage{tikz-cd}
\usepackage[utf8]{inputenc}
\usepackage[T1]{fontenc}

\usepackage{relsize}
\usepackage[bbgreekl]{mathbbol}
\usepackage{amsfonts}

\DeclareSymbolFontAlphabet{\mathbb}{AMSb}
\DeclareSymbolFontAlphabet{\mathbbl}{bbold}

\usepackage{enumitem}
\usepackage[colorlinks=true,hyperindex,pagebackref,linktocpage=true]{hyperref}
\usepackage[nameinlink]{cleveref}
\hypersetup{
	colorlinks,
	linkcolor={red!40!black},
	citecolor={cyan!70!black},
	urlcolor={orange!60!black}
}

\newtheorem{thm}{Theorem}[section]
\newtheorem{thm2}{Theorem}

\newtheorem*{conj2}{Conjecture}
\newtheorem{conj}{Conjecture}

\newtheorem{prop}[thm]{Proposition}

\newtheorem{cor}[thm]{Corollary}
\newtheorem{lem}[thm]{Lemma}

\theoremstyle{definition}
\newtheorem{defi}[thm]{Definition}

\newcommand{\bb}[1]{\mathbb{#1}}
\newcommand{\cl}[1]{{\mathcal{#1}}}

\newcommand{\mfr}[1]{{\mathfrak{#1}}}
\newcommand{\mrm}[1]{{\mathrm{#1}}}

\newcommand{\ov}[1]{{\overline{#1}}}
\newcommand{\wtd}[1]{{\widetilde{#1}}}

\newcommand{\Zl}{{\mathbb{Z}_\ell}}

\newcommand{\Qp}{{\mathbb{Q}_p}}
\newcommand{\Qla}{{\overline{\mathbb{Q}}_\ell}}
\newcommand{\LG}{{{}^LG}}
\newcommand{\LP}{{{}^LP}}
\newcommand{\LU}{{{}^LU}}
\newcommand{\LM}{{{}^LM}}
\newcommand{\Gm}{{\mathbb{G}_m}}
\newcommand{\Ga}{{\mathbb{G}_a}}
\newcommand{\Al}{\mathbb{A}^1}

\newcommand{\invp}{{\tfrac{1}{p}}}
\newcommand{\st}{{\hspace{4pt} \mathrm{s.t.} \hspace{4pt}}}

\newcommand{\Bun}{{\operatorname{Bun}}}
\newcommand{\dtrg}{{\operatorname{dim.trg}}}
\newcommand{\Eis}{{\operatorname{Eis}}}
\newcommand{\CT}{{\operatorname{CT}}}
\newcommand{\lis}{{\operatorname{lis}}}
\newcommand{\nuc}{{\operatorname{nuc}}}
\newcommand{\un}{{\operatorname{un}}}
\newcommand{\qc}{{\operatorname{qc}}}
\newcommand{\et}{{\acute{e}t}}
\newcommand{\Fil}{{\operatorname{Fil}}}
\newcommand{\gr}{{\operatorname{gr}}}
\newcommand{\spec}{{\operatorname{spec}}}
\newcommand{\bas}{{\operatorname{bas}}}

\newcommand{\id}{{\operatorname{id}}}
\newcommand{\fib}{{\operatorname{fib}}}
\newcommand{\cofib}{{\operatorname{cofib}}}
\newcommand{\pr}{{\operatorname{pr}}}

\newcommand{\colim}{\operatorname{colim}}
\newcommand{\BZ}{{\operatorname{BZ}}}
\newcommand{\Hom}{{\operatorname{Hom}}}
\newcommand{\RHom}{{\operatorname{RHom}}}

\newcommand{\Ind}{{\operatorname{Ind}}}

\newcommand{\Par}{{\operatorname{Par}}}
\newcommand{\IndCoh}{{\operatorname{IndCoh}}}
\newcommand{\Nilp}{{\operatorname{Nilp}}}
\newcommand{\End}{{\operatorname{End}}}
\newcommand{\ev}{{\operatorname{ev}}}
\newcommand{\ssi}{{\operatorname{ss}}}
\newcommand{\cusp}{{\operatorname{cusp}}}
\newcommand{\Sing}{{\operatorname{Sing}}}
\newcommand{\Fr}{{\operatorname{Fr}}}
\newcommand{\SL}{{\operatorname{SL}}}

\newcommand{\FS}{{\operatorname{FS}}}
\newcommand{\Rep}{{\operatorname{Rep}}}
\newcommand{\GL}{{\operatorname{GL}}}
\newcommand{\GSp}{{\operatorname{GSp}}}
\newcommand{\std}{{\operatorname{std}}}
\newcommand{\Sht}{{\operatorname{Sht}}}
\newcommand{\sm}{{\operatorname{sm}}}
\newcommand{\Spa}{{\operatorname{Spa}}}
\newcommand{\Ad}{{\operatorname{Ad}}}
\newcommand{\Mant}{{\operatorname{Mant}}}

\title{Second adjointness and cuspidal supports at the categorical level}
\author{Yuta Takaya}
\address{Graduate School of Mathematical Sciences, The University of Tokyo, 3-8-1 Komaba,
Meguro-ku, Tokyo 153-8914, Japan}
\email{takaya@ms.u-tokyo.ac.jp}
\begin{document}
\begin{abstract}
	We prove the second adjointness in the setting of the categorical local Langlands correspondence. Moreover, we study the relation between Eisenstein series and  cuspidal supports and present a conjectural characterization of irreducible smooth representations with supercuspidal $L$-parameters regarding geometric constant terms. 
	The main technical ingredient is an induction principle for geometric Eisenstein series which allows us to reduce to the situations already treated in the literature.  
\end{abstract}

\maketitle
\tableofcontents

\section{Introduction}

In the remarkable work of \cite{FS21}, Fargues and Scholze established geometric frameworks for the local Langlands correspondence and proposed a categorical enhancement of the correspondence that can be regarded as the geometric Langlands correspondence over the Fargues--Fontaine curve. 

Let us briefly recall the setting of the categorical local Langlands correspondence. 
Let $E$ be a non-archimedean local field of residual characteristic $p$ and let $G$ be a quasi-split connected reductive group over $E$. 
Let $q$ be the cardinality of the residue field of $E$ and let $\Lambda$ be a $\Zl[\sqrt{q}]$-algebra which plays the role of a coefficient ring. 
The key player on the automorphic side is the moduli stack $\Bun_G$ of $G$-bundles over the Fargues--Fontaine curve and the key player on the spectral side is the moduli stack $\Par_\LG$ of $L$-parameters of $G$ over $\Lambda$-algebras. 
The categorical local Langlands correspondence expects the existence of a categorical equivalence
\[\cl{D}(\Bun_G,\Lambda)\cong \IndCoh_\Nilp(\Par_\LG)\]
determined by a choice of a Whittaker datum. 
Here, $\cl{D}(\Bun_G,\Lambda)$ is the derived category of suitable sheaves of $\Lambda$-modules on $\Bun_G$ and $\IndCoh_\Nilp(\Par_\LG)$ is the derived category of ind-coherent sheaves on $\Par_\LG$ with nilpotent singular supports. 

As in the geometric Langlands correspondence, geometric Eisenstein series are expected to play the role of parabolic inductions at the categorical level. In the context of the categorical local Langlands correspondence, they were first introduced in \cite{Ham23} and studied along the lines of \cite{BG02}. 
In particular, geometric Eisenstein series are used to construct Hecke eigensheaves of sufficiently regular $L$-parameters factoring through maximal tori. 

In this paper, we study the second adjointness and cuspidal supports in the context of the categorical local Langlands correspondence. First, let us comment on the second adjointness. Let $P$ be a parabolic subgroup of $G$ with a Levi subgroup $M$. Let $\ov{P}$ be the parabolic subgroup of $G$ opposite to $P$. Then, we have the following diagram. 
\[
\begin{tikzcd}
    & \Bun_P \arrow[ld, "\mfr{p}"']  \arrow[rd, "\mfr{q}"] & \\
    \Bun_G & & \Bun_M \\
    & \Bun_\ov{P} \arrow[lu, "\ov{\mfr{p}}"']  \arrow[ru, "\ov{\mfr{q}}"] & 
\end{tikzcd}
\]
The unnormalized geometric Eisenstein series of $P$ is denoted by $\Eis^\un_P=\mfr{p}_!\mfr{q}^*$ and the unnormalized geometric constant term of $\ov{P}$ is denoted by $\CT^\un_\ov{P}=\ov{\mfr{q}}_!\ov{\mfr{p}}^*$. Here, we need a six-functor formalism of sheaves of $\Lambda$-modules on $\Bun_G$ to define these functors and at this moment this definition makes sense only when we work with torsion \'{e}tale sheaves. Our form of the second adjointness is as follows. 

\begin{thm2}\textup{(\Cref{thm:unEis})}\label{thm2:secadj}
	For a torsion $\bb{Z}[\invp]$-algebra $\Lambda$, $\Eis_P^\un$ is left adjoint to $\CT_\ov{P}^\un$ and preserves compact objects. 
\end{thm2}

Note that this second adjointness implies the compatibility of geometric Eisenstein series with the Bernstein--Zelevinsky involution as noted in \cite[Exercise 1.5.4]{Han24}. 

Since $\ov{\mfr{q}}$ is $\ell$-cohomologically smooth, the definition of $\CT_\ov{P}^\un$ makes sense even if we work with lisse-\'{e}tale sheaves and it preserves colimits by definition. We show in \Cref{prop:adjlis} that $\CT_\ov{P}^\un$ also preserves limits and admits a left adjoint. When we work with lisse-\'{e}tale sheaves, we define geometric Eisenstein series $\Eis_P^\un$ just as the left adjoint of $\CT_\ov{P}^\un$. In certain cases, we can verify their descriptions explicitly as in \Cref{lem:explicit}. 

Our proof of \Cref{thm2:secadj} is different from the proof of \cite{DG16} in the geometric setting. We explain the difference in more detail in \Cref{ssc:DG}. Our proof is based on the decomposition of the correspondence along $\Bun_P$ and uses the induction principle established in \Cref{ssc:ind}. This reduces the proof of \Cref{thm2:secadj} to the second adjointness established in \cite{DHKM24} and hyperbolic localization for diamonds established in \cite{FS21}. In particular, our proof of the second adjointness relies on the representation-theoretic counterpart. 
After releasing a preprint of this paper, Linus Hamann informed the author that he, David Hansen, and Peter Scholze are also working on a paper with an almost identical argument on second adjointness.

In the second part of this paper, we study cuspidal supports in the categorical setting when $\Lambda=\Qla$. 
Let $W_E$ be the Weil group of $E$. For every semisimple $L$-parameter $\varphi\colon W_E\to \LG(\Qla)$, there is a unique (up to conjugacy) Levi subgroup $M$ of $G$ such that $\varphi$ is $\widehat{G}$-conjugate to a supercuspidal $L$-parameter of $M$. We call this conjugacy class of $M$, denoted by $[M]$, the cuspidal support of $\varphi$. This defines a decomposition of the coarse moduli space of $\Par_\LG$ into closed and open subspaces and thus we have a decomposition $\Par_\LG=\coprod_{[M]} \Par_{\LG,[M]}$. Each $\Par_{\LG,[M]}$ classifies those $L$-parameters of $G$ such that the cuspidal supports of their semisimplications are $[M]$. The spectral action induces a corresponding decomposition $\cl{D}_\lis(\Bun_G,\Qla)=\bigoplus_{[M]}\cl{D}_\lis(\Bun_G,\Qla)_{[M]}$ on the automorphic side. 
We study the relation between $\cl{D}_\lis(\Bun_G,\Qla)_{[M]}$ and $\cl{D}_\lis(\Bun_M,\Qla)_{[M]}$ through geometric Eisenstein series. 
For this, it is easier to work with normalized geometric Eisenstein series. 
Let $K_{\Bun_P}$ be the dualizing complex of $\Bun_P$. It is explicitly calculated in \cite{HI24} and since it can be defined over $\Bun_M$, we regard it as a lisse-\'{e}tale sheaf over $\Bun_M$. Choose a square root of $q$ in $\Qla$. Then, it defines a canonical square root $K_{\Bun_P}^{1/2}$ and we can define the normalized geometric Eisenstein series as $\Eis_P=\Eis^\un_P((-)\otimes K_{\Bun_P}^{1/2})$. 
In \cite[Exercise 1.5.5]{Han24}, it is expected to commute with the action of the spectral Bernstein center. In this paper, we verify a bit weaker form of this expectation as follows. 

Let $\cl{Z}^\spec(G,\Lambda)$ be the ring of global sections on $\Par_\LG$. This is what we call the spectral Bernstein center of $G$. 
For an element $f\in \cl{Z}^\spec(G,\Lambda)$, its restriction to $\cl{Z}^\spec(M,\Lambda)$ is denoted by $f^M$. 

\begin{thm2}\textup{(\Cref{thm:compat})}\label{thm2:compat}
	For every $A\in \cl{D}_\lis(\Bun_M,\Lambda)$ with a quasicompact support, there is a finite filtration $\Fil^\bullet\Eis_P(A)$ of $\Eis_P(A)$ with an action of $\cl{Z}^\spec(M,\Lambda)$ extending the spectral action on $\Eis_P(A)$ such that for every $f\in \cl{Z}^\spec(G,\Lambda)$, the action of $f^M$ on each subquotient $\gr^{\bullet}(\Eis_P(A))$ is equal to the spectral action of $f$.  
\end{thm2}

This proof is also based on the induction principle established in \Cref{ssc:ind} and ultimately reduces to \cite[Theorem IX.7.2]{FS21} and \cite[Corollary IX.7.3]{FS21}. We can only prove the commutativity up to finite filtrations in this manner because we need to take excision sequences through induction steps. 

As a corollary of \Cref{thm2:compat}, it follows that idempotents in the spectral Bernstein center commute with normalized geometric Eisenstein series. Thus, we can show that $\Eis_P$ sends $\cl{D}_\lis(\Bun_M,\Qla)_{[M]}$ to $\cl{D}_\lis(\Bun_G,\Qla)_{[M]}$. In \Cref{ssc:specEis}, we study its spectral counterpart and show that $\Eis^\spec_{P}$ sends $\IndCoh(\Par_{\LM,[M]})$ to $\IndCoh(\Par_{\LG,[M]})$ and moreover the essential images of $\Eis^\spec_P$ for all $P$ generate the target category by following the argument in \cite{AG15}. Thus, it is expected that the same holds on the automorphic side and this expectation can be phrased in a different form as follows. 

Let $\cl{D}_\lis(\Bun_G,\Qla)_\cusp$ be the subcategory of $\cl{D}_\lis(\Bun_G,\Qla)$ consisting of the objects $A$ such that $\CT_P(A)=0$ for every proper parabolic subgroup $P$ of $G$. Here, $\CT_P$ is the normalized geometric constant term functor. It is easy to see from \Cref{thm2:compat} that $\cl{D}_\lis(\Bun_G,\Qla)_{[G]}\subset \cl{D}_\lis(\Bun_G,\Qla)_\cusp$. Now, we can see that the converse inclusion is equivalent to the essential surjectivity of $\Eis_P$ from $\cl{D}_\lis(\Bun_M,\Qla)_{[M]}$ to $\cl{D}_\lis(\Bun_G,\Qla)_{[M]}$ for all $P$. Thus, we derive the following conjecture.  
For $b\in B(G)_\bas$, let $G_b$ be the extended pure inner form of $G$ corresponding to $b$ and $i^b\colon \Bun_G^b=[\ast/G_b(E)]\hookrightarrow \Bun_G$ be the open immersion. 

\begin{conj2}\textup{(\Cref{conj:cusp})}\label{conj2:cusp}
	The inclusion $\cl{D}_\lis(\Bun_G,\Qla)_{[G]} \subset \cl{D}_\lis(\Bun_G,\Qla)_\cusp$ is an equality. In particular, for any smooth irreducible $\Qla$-representation $\pi$ of $G_b(E)$ with $b\in B(G)_\bas$, its Fargues--Scholze parameter $\varphi^\FS_\pi$ is supercuspidal if and only if $\CT_P(i^b_{!}\pi)=0$ for every proper parabolic subgroup $P$ of $G$. 
\end{conj2}

This characterization of irreducible smooth representations in supercuspidal $L$-packets is close in spirit to the representation-theoretic characterization of supercuspidal $L$-packets in \cite[\S 3.5]{DR09}. 
In \Cref{sssc:GLn} and \Cref{sssc:GSp4}, we study the cases where $G=\GL_n, \GSp_4$. 
What we do there to supercuspidal representations $\pi$ of $G_b(E)$ with non-supercuspidal $L$-parameters is to apply a suitable Hecke operator $T_V$ to $i^b_!\pi$ and extract an irreducible non-supercuspidal constituent from $T_Vi_!^b \pi$. Considering the expectation that Hecke operators should enumerate every member of the $L$-packet containing $\pi$, we can see the connection to the classical characterization in \cite[\S 3.5]{DR09} which expects that non-supercuspidal $L$-packets should contain non-supercuspidal representations. 

\subsection{The structure of the paper}
In \Cref{sec:secadj}, we describe our induction principle and deduce the second adjointness. In \Cref{sec:spact}, we first study spectral Eisenstein series and then study the relation between cuspidal supports and geometric Eisenstein series. 

\subsection*{Acknowledgements}
I would like to thank my advisor Yoichi Mieda for his constant support and encouragement. I am also grateful to Linus Hamann for informing me of their work on geometric Eisenstein series and explaining the difference between isocrystal and Harder--Narasimhan slopes of $\sigma$-conjugacy classes. 
This work was supported by the WINGS-FMSP program at the Graduate School of Mathematical Sciences, the University of Tokyo and JSPS KAKENHI Grant number JP24KJ0865.

\subsection*{Notation}

All rings are assumed to be commutative. 
We fix a prime $p$ and another prime $\ell \neq p$.
Let $E$ be a non-archimedean local field of residual characteristic $p$ with the ring of integers $O_E$. 
Let $k$ be the residue field of $O_E$ and $q$ be the cardinality of $k$. 
Frobenius maps relative to $k$ (taking the $q$-th power) are denoted by $\sigma$.
The absolute Galois group of $E$ is denoted by $\Gamma$. 
The $v$-stacks are considered as stacks over the category of perfectoid spaces over $k$ with the $v$-topology. 

\section{Second adjointness}\label{sec:secadj}
\subsection{Moduli stacks of bundles}\label{ssc:BunG}
In this section, we set up our notation used throughout this paper and review the properties of the moduli stacks of bundles over the Fargues--Fontaine curves established in \cite{FS21}. 

For an algebraic group $H$ over $E$, the moduli $v$-stack of $H$-bundles over the Fargues--Fontaine curve is denoted by $\Bun_H$. Now, we suppose that $H$ is connected and reductive. Then, $\Bun_H$ is an Artin $v$-stack of dimension $0$ and its underlying space $\lvert \Bun_H \rvert$ is homeomorphic to the set of $\sigma$-conjugacy classes $B(H)$ with the partial order given by Newton points and Kottwitz maps. The semistable locus $\Bun_H^\ssi$ is an open substack of $\Bun_H$ and the corresponding open subset of $B(H)$ consists of basic $\sigma$-conjugacy classes $B(H)_\bas$. For each $b\in B(H)$, the corresponding locally closed substack of $\Bun_H$ is denoted by $\Bun_H^b$ and the open substack of $\Bun_H$ consisting of $\sigma$-conjugacy classes under $b$ with respect to the partial order is denoted by $\Bun_H^{\preceq b}$. The locally closed immersion $\Bun_H^b\hookrightarrow \Bun_H$ is denoted by $i^b$. The Newton point of $b$ is denoted by $\nu_b$. 

Let $G$ be a quasi-split connected reductive group over $E$. Let $\Lambda$ be a coefficient ring where $p$ is invertible. 
The key player on the automorphic side is the derived category of sheaves of $\Lambda$-modules on $\Bun_G$. 
When $\Lambda$ is torsion, we work with the derived category $\cl{D}_\et(\Bun_G,\Lambda)$ of \'{e}tale sheaves of $\Lambda$-modules on $\Bun_G$ and when $\Lambda$ is a $\Zl$-algebra, we work with the derived category $\cl{D}_\lis(\Bun_G,\Lambda)$ of lisse-\'{e}tale sheaves of $\Lambda$-modules on $\Bun_G$.

Let $P$ be a parabolic subgroup of $G$ with a Levi subgroup $M$. The geometric Eisenstein series is the functor associated with the following correspondence. 
\[
\begin{tikzcd}
    & \Bun_P \arrow[ld, "\mfr{p}"']  \arrow[rd, "\mfr{q}"] & \\
    \Bun_G & & \Bun_M
\end{tikzcd}
\]
Here, $\mfr{p}$ is representable in locally spatial diamond, compactifiable and locally $\dtrg<\infty$ and $\mfr{q}$ is $\ell$-cohomologically smooth (see \cite[Lemma 4.1]{ALB21}). We recall geometric Eisenstein series more precisely in \Cref{ssc:EisCT}. 
Here, we note that this correspondence is associative in the sense that for a parabolic subgroup $P'\supset P$ of $G$ with a Levi subgroup $M'\supset M$, the diagram
\[
	\begin{tikzcd}
		& & \Bun_P \arrow[ld]  \arrow[rd] & & \\
		& \Bun_{P'} \arrow[ld]  \arrow[rd] & & \Bun_{M'\cap P} \arrow[ld] \arrow[rd]& \\
		\Bun_G & & \Bun_{M'} & & \Bun_M
	\end{tikzcd}
\]
describes the composition of the correspondences for $P'$ and $M'\cap P$. 

In \Cref{ssc:ind}, we will ``decompose'' the above correspondence into certain basic correspondences. Here, we describe the constituents of this decomposition. 
Let $b\in B(M)_\bas$. The preimage $\mfr{q}^{-1}(\Bun_M^b)$ is denoted by $\Bun_P^b$. The constituents of our decomposition are the correspondences from $\Bun_M^b$ to $\Bun_G$ along $\Bun_P^b$ which satisfy one of the following conditions. 
\begin{enumerate}
	\item $\nu_b$ is strictly $P$-codominant, 
	\item $\nu_b$ is central in $G$, or equivalently $b$ is basic in $G$, or
	\item $\nu_b$ is strictly $P$-dominant. 
\end{enumerate}
Here, $\nu_b$ is said to be strictly $P$-dominant (resp.\ strictly $P$-codominant) if for every positive root $\alpha$ in $P$ outside $M$, we have $\langle \alpha, \nu_b \rangle >0$ (resp.\ $\langle \alpha, \nu_b \rangle <0$). If $P$ is a maximal parabolic subgroup of $G$, then there is only one simple root outside $M$, so every $b\in B(G)_\bas$ is classified into one of the three cases. 

In case (1), we have $\Bun_P^b=\Bun_G^b$ and $\mfr{p}\vert_{\Bun_P^b}=i^b$. In case (2), the associated correspondence goes from $[\ast/M_b(E)]$ to $[\ast/G_b(E)]$ and it is equal to the representation-theoretic correspondence of the parabolic induction along $P_b(E)$. Only the correspondence in case (3) is something new and what we come across here is the local chart $\cl{M}_b$ introduced in \cite[Section V.3]{FS21}. In this case, we have $\Bun_P^b=\cl{M}_b$ and $\mfr{p}\vert_{\cl{M}^b}$ is $\ell$-cohomologically smooth. In all of these cases, the correspondences factor through the open substack $\Bun_G^{\preceq b}$.

\subsection{An induction principle}\label{ssc:ind}
In this section, we give an induction principle for geometric Eisenstein series and constant terms by decomposing the correspondence along $\Bun_P$. 
This decomposition is obtained through successive combinations of excision and factorization into two correspondences. 
The following proposition allows us to reduce a proof of certain properties of parabolic inductions to the basic cases introduced in \Cref{ssc:BunG}. 

\begin{prop} \label{prop:ind}
	Let $\Psi_{(G,P,b)}$ be a property of a parabolic subgroup $P$ of a quasi-split connected reductive group $G$ over $E$ and a $\sigma$-conjugacy class $b\in B(M)$ of a Levi subgroup $M$ of $P$.
    Suppose that $\Psi_{(G,P,b)}$ holds if one of the following conditions holds. 

    \begin{enumerate}
        \item $b$ is basic and $\nu_b$ is strictly $P$-codominant. 
        \item $b$ is basic and $P$ is a maximal proper parabolic subgroup of $G$. 
        \item $\Psi_{(G,P',b')}$ holds for every pair $(P',b')$ such that $P'\subset P$, $b'$ is basic, $\nu_{b'}$ is strictly $(M\cap P')$-codominant and
        $b'$ is a generalization of $b$ in $B(M)$. 
        \item Both $\Psi_{(G,P',b)}$ and $\Psi_{(M',M'\cap P,b)}$ hold for some parabolic subgroup $P'\supset P$ with a Levi subgroup $M'\supset M$.  
    \end{enumerate}

    Then, $\Psi_{(G,P,b)}$ holds for every triple $(G,P,b)$. 
\end{prop}

Before giving a proof, let us explain how to apply this proposition.
In practice, $\Psi_{(G,P,b)}$ indicates a property of the parabolic induction along $\Bun_P$ from the open substack $\Bun_M^{\preceq b}\subset \Bun_M$. 
Condition (3) means that $\Psi_{(G,P,b)}$ is excisive and condition (4) means that $\Psi_{(G,P,b)}$ is associative. 
Under these conditions, the proposition states that it is enough to check the property $\Psi_{(G,P,b)}$ for basic correspondences. 

\begin{proof}
    We fix a quasi-split connected reductive group $G$ and a Borel pair $B \supset T$. Here, we always assume that $P\supset B$ and $M\supset T$. For every $b\in B(M)$, there is a parabolic subgroup $P'\subset P$ such that $b=[b']$ for some $b'\in B(M')_\bas$ with $\nu_{b'}$ strictly $(M\cap P')$-codominant. Thus, by condition (3), it is enough to show $\Psi_{(G,P,b)}$ for every pair $(P,b)$ with $b\in B(M)_\bas$. 

    Let $(P,b)$ be a pair with $b\in B(M)_\bas$. If $\nu_b$ is not strictly $P$-codominant, there is a simple root $\alpha$ of $P$ outside $M$ such that $\langle \alpha,\nu_b \rangle \geq 0$. 
    Let $P_\alpha$ be the minimal parabolic subgroup containing $P$ whose Levi subgroup contains $\alpha$ as a root. Let $M_\alpha$ be the Levi subgroup of $P_\alpha$. 
	Condition (2) states that $\Psi_{(M_\alpha,M_\alpha\cap P,b)}$ holds, so it is enough to show $\Psi_{(G,P_\alpha,b)}$ by condition (4). 
    If $\langle \alpha,\nu_b \rangle=0$, then $b$ is basic in $M_\alpha$. 
    If $\langle \alpha, \nu_b \rangle >0$, then for every $b' \preceq b$ in $B(M_\alpha)$, we have $w_0^{M_\alpha}\nu_b < \nu_{b'} \leq \nu_b$. Here, $w_0^{M_\alpha}$ is the longest element in the Weyl group of $M_\alpha$. The second inequality follows from the description of the partial order on $B(M_\alpha)$ and the first one follows from \cite[Lemma 2.2]{RR96} since  $w_0^{M_\alpha}\nu_b$ is $M_\alpha$-codominant. 
    Thus, condition (3) reduces $\Psi_{(G,P,b)}$ to $\Psi_{(G,P',b')}$ for pairs $(P',b')$ with $w_0^{M_\alpha}\nu_b \leq \nu_{b'} <\nu_b$ since $\nu_{b'}$ is $(M_\alpha\cap P')$-codominant and the Newton point of $b'$ as an element of $B(M_\alpha)$ is equal to $w_0^{M_\alpha}\nu_{b'}$. 

	Consider repeating this reduction until all remaining pairs $(P,b)$ satisfy that $\nu_b$ is strictly $P$-codominant. In one step, the semisimple rank of $M$ increases or the Newton point of $b$ decreases. Moreover, the Newton point of $b$ does not increase and remains in a bounded subset. In particular, $\nu_b$ can decrease only finitely many times, so the above reduction stops after a finite number of repetitions. Thus, $\Psi_{(G,P,b)}$ follows from condition (1). 
\end{proof}

\subsection{Geometric Eisenstein series and geometric constant terms}\label{ssc:EisCT}
In this section, we establish the second adjointness via the induction principle established in \Cref{prop:ind}. 
First, we review the definition of geometric Eisenstein series and geometric constant term functors with torsion coefficients. We will use the following notation. 
\begin{defi}\label{def:etEis}
	Let $\Lambda$ be a torsion $\bb{Z}[\invp]$-algebra. The unnormalized geometric Eisenstein series functor, the unnormalized geometric constant term functor and their $\ast$-variants with coefficients in $\Lambda$ are defined as
	\begin{align*}
		\Eis^\un_P&=\mfr{p}_!\mfr{q}^*\colon \cl{D}_\et(\Bun_M,\Lambda)\to \cl{D}_\et(\Bun_G,\Lambda), \\
		\CT^\un_P&=\mfr{q}_!\mfr{p}^*\colon \cl{D}_\et(\Bun_G,\Lambda)\to \cl{D}_\et(\Bun_M,\Lambda), \\
		\Eis^\un_{P*}&=\mfr{p}_*\mfr{q}^!\colon \cl{D}_\et(\Bun_M,\Lambda)\to \cl{D}_\et(\Bun_G,\Lambda), \\
		\CT^\un_{P*}&=\mfr{q}_*\mfr{p}^!\colon \cl{D}_\et(\Bun_G,\Lambda)\to \cl{D}_\et(\Bun_M,\Lambda).
	\end{align*}
	In particular, $\CT^\un_{P*}$ is a right adjoint of $\Eis^\un_P$ and $\Eis^\un_{P*}$ is a right adjoint of $\CT^\un_P$. 
\end{defi}
As in \cite[Lemma 4.17]{HI24}, these functors are associative. Let $\ov{P}$ be the parabolic subgroup of $G$ opposite to   $P$ and consider the following diagram. 
\[
\begin{tikzcd}
    & \Bun_P \arrow[ld, "\mfr{p}"']  \arrow[rd, "\mfr{q}", shift left=0.5ex] & \\
    \Bun_G & & \Bun_M \arrow[lu, "\mfr{i}", shift left=0.5ex] \arrow[ld,"\ov{\mfr{i}}", shift left=0.5ex]\\
    & \Bun_\ov{P} \arrow[lu, "\ov{\mfr{p}}"']  \arrow[ru, "\ov{\mfr{q}}", shift left=0.5ex] & 
\end{tikzcd}
\]
We will show that $\Eis_P^\un$ is left adjoint to $\CT_{\ov{P}}^\un$. This is what we call the second adjointness at the categorical level. 
It has been observed that the second adjointness is highly related to hyperbolic localization (\cite{DG16}, \cite{Bra03}) and we will rely on hyperbolic localization for diamonds established in \cite[Section IV.6]{FS21}. However, our approach is different from \cite{DG16} and we will explain the difference in more detail in \Cref{ssc:DG}. 

First, we follow the usual construction in the theory of hyperbolic localization to get a natural transformation $\CT^\un_{\ov{P}*}\to \CT^\un_P$. 

\begin{lem}\label{lem:nattr}
	For a torsion $\bb{Z}[\invp]$-algebra $\Lambda$, there is a natural transformation $\CT^\un_{\ov{P}*}\to \CT^\un_{P}$. 
\end{lem}
\begin{proof}
	Composing with the unit $\id\to \ov{\mfr{i}}_*\ov{\mfr{i}}^*$ and the counit $\mfr{i}_!\mfr{i}^!\to \id$, we obtain natural transformations $\CT^\un_{\ov{P}*}\to \ov{\mfr{i}}^*\ov{\mfr{p}}^!$ and $\mfr{i}^!\mfr{p}^*\to \CT^\un_P$, so it is enough to construct $\ov{\mfr{i}}^*\ov{\mfr{p}}^! \to \mfr{i}^!\mfr{p}^*$. 
	Let $\mfr{j}\colon \Bun_M\to \Bun_P\times_{\Bun_G} \Bun_\ov{P}$ be the map induced by $\mfr{i}$ and $\ov{\mfr{i}}$. Let $\pr_i$ ($i=1,2$) be the projection of $\Bun_P\times_{\Bun_G} \Bun_\ov{P}$ to the $i$-th component. 
	Since there is a natural transformation $\pr_2^*\ov{\mfr{p}}^!\to \pr_1^!\mfr{p}^*$, it is enough to show that $\mfr{j}$ is an open immersion. 

	Let $S$ be an affinoid perfectoid space over $k$ and $\cl{E}$ be a $G$-torsor over the schematic Fargues--Fontaine curve $X_S$. 
	Following the notation in \cite[IV.3]{FS21}, the fiber $\Bun_P\times_{\Bun_G} S$ over $\cl{E}$ is isomorphic to $\cl{M}_{\cl{E}/P}$ and the fiber of $\mfr{j}$ over $\cl{E}$ is induced by the morphism $\cl{E}/M\to \cl{E}/P\times \cl{E}/\ov{P}$. 
	Since $\cl{E}/M$ can be regarded as a big cell inside $\cl{E}/P\times \cl{E}/\ov{P}$, in a sense that it consists of pairs $(e_1,e_2)$ such that $e_1^{-1}e_2$ lies in a big cell $P\ov{P}\subset G$, the morphism $\cl{E}/M\to \cl{E}/P\times \cl{E}/\ov{P}$ is an open immersion. 
	It induces an open immersion $\cl{M}_{\cl{E}/M}\to \cl{M}_{\cl{E}/P\times \cl{E}/\ov{P}}$, so $\mfr{j}$ is an open immersion. 
\end{proof}

It is easy to see that this transformation is associative. 
We compare this map with representation-theoretic counterparts on the semistable locus. 

\begin{lem}\label{lem:repsec}
	Let $b\in B(M)$ be a basic $\sigma$-conjugacy class which is also basic in $G$. 
	For a torsion $\bb{Z}[\invp]$-algebra $\Lambda$, the natural transformation $\CT^\un_{\ov{P}*}\to \CT^\un_{P}$ is an isomorphism on $\Bun^b_M$.
\end{lem}
\begin{proof}
	The correspondence over $\Bun_M^b$ is written as $[\ast/G_b(E)]\xleftarrow{\mfr{p}_b} [\ast/P_b(E)] \xrightarrow{\mfr{q}_b} [\ast/M_b(E)]$. 
	Let us give the representation-theoretic description of the derived functors appearing in this correspondence. 
	First, $\mfr{p}_b^*$ corresponds to the restriction of representations of $G_b(E)$ to $P_b(E)$ and $\mfr{q}_b^*$ corresponds to the inflation of representations of $M_b(E)$ to $P_b(E)$.
	Next, $\mfr{p}_{b!}$ corresponds to the induction of representations of $P_b(E)$ to $G_b(E)$. This is because ${(G_b/P_b)(E)}\to *$ is proper and the pushforward of locally constant sheaves over profinite sets is the set of global sections by \cite[Proposition 14.9]{Sch17}. 
	Finally, since the dualizing complex of $[\ast/P_b(E)]$ is the modulus character $\delta_{P_b}$, $\mfr{q}_{b!}$ corresponds to the $\delta_{P_b}^{-1}$-twisted $N_b(E)$-coinvariant with $N_b$ the unipotent radical of $P_b$. 
	Thus, $\Eis^\un_P\vert_{\Bun_M^b}=\mfr{p}_{b!}\mfr{q}_b^*$ is identified with the unnormalized parabolic induction $\sigma \mapsto \Ind^{G_b}_{P_b}(\sigma)$ and $\CT^\un_P\vert_{\Bun_M^b}=\mfr{q}_{b!}\mfr{p}_b^*$ is identified with the twisted Jacquet functor $\pi\mapsto (\pi\vert_{P_b} \otimes \delta_{P_b}^{-1})_{N_b}$.
	We see from \cite{DHKM24} that $\Eis^\un_\ov{P}\vert_{\Bun_M^b}$ is left adjoint to $\CT^\un_P\vert_{\Bun_M^b}$ and $\CT^\un_{\ov{P}*}$ is isomorphic to $\CT^\un_{P}$ on $\Bun^b_M$. 
	It remains to verify that this isomorphism is equal to the map constructed in \Cref{lem:nattr}.  

	For a smooth representation $\pi$ of $G_b(E)$, the transformation in \Cref{lem:nattr} is identified with the composition $\pi^{\ov{N}_b(E)}\to \pi\to \pi_{N_b(E)}$. 
	The second adjointness established in \cite{DHKM24} is deduced from a pairing between $\pi_{N_b(E)}$ and $\pi^\vee_{\ov{N}_b(E)}$.
	Since its restriction to $\pi^{\ov{N}_b(E)}$ is induced from the natural pairing between $\pi$ and $\pi^\vee$, we see that the transformation in \Cref{lem:nattr} matches the isomorphism coming from the second adjointness established in \cite{DHKM24}. 
\end{proof}

\begin{thm}\label{thm:unEis}
	For a torsion $\bb{Z}[\invp]$-algebra $\Lambda$, the natural transformation $\CT^\un_{\ov{P}*}\to \CT^\un_P$ is an isomorphism. 
	In particular, $\Eis^\un_P$ is left adjoint to $\CT^\un_{\ov{P}}$ and preserves compact objects. 
\end{thm}

\begin{proof}
	We apply \Cref{prop:ind} to the property $\Psi_{(G,P,b)}$ that the natural transformation $\CT^\un_{\ov{P}*}\to \CT^\un_P$ is an isomorphism on the open substack $\Bun_M^{\preceq b}\subset \Bun_M$. 
	Since the formation of $\CT^\un_P$ and $\CT^\un_{\ov{P}*}$ is associative and $\mfr{p}(\mfr{q}^{-1}(\lvert \Bun_M^{\preceq b} \rvert))$ is contained in $\lvert \Bun_G^{\preceq b} \rvert$ (see the proof of \cite[Proposition V.3.6]{FS21}), $\Psi_{(G,P,b)}$ is associative. 
	Next, for $b\in B(M)_\bas$, we show $\Psi_{(G,P,b)}$ in the following cases: (1) $\nu_b$ is strictly $P$-codominant, (2) $b$ is basic in $G$, and (3) $\nu_b$ is strictly $P$-dominant. 
	The case (2) is treated in \Cref{lem:repsec}. 
	To treat the cases (1) and (3), suppose that $\nu_b$ is strictly $P$-codominant and consider the following diagram. 
	\[
	\begin{tikzcd}
		& \Bun_G^b \arrow[ld, "\mfr{p}_b"']  \arrow[rd, "\mfr{q}_b", shift left=0.5ex] & \\
		\Bun_G & & \lbrack \ast/M_b(E) \rbrack \arrow[lu, "\mfr{i}_b", shift left=0.5ex] \arrow[ld,"\ov{\mfr{i}}_b", shift left=0.5ex]\\
		& \cl{M}_b \arrow[lu, "\ov{\mfr{p}}_b"']  \arrow[ru, "\ov{\mfr{q}}_b", shift left=0.5ex] & 
	\end{tikzcd}
	\]
	It is enough to show that the maps $\ov{\mfr{q}}_{b*}\ov{\mfr{p}}_b^!\to \mfr{q}_{b!}\mfr{p}_b^*$ and $\mfr{q}_{b*}\mfr{p}_b^!\to \ov{\mfr{q}}_{b!} \ov{\mfr{p}}_b^*$ are isomorphisms. 
	First, by \cite[Proposition V.2.2]{FS21}, $\mfr{i}_b^*$ and $\mfr{i}_b^!$ are equivalences with the inverses $\mfr{q}_b^*$ and $\mfr{q}_b^!$, so the maps $\mfr{q}_{b*}\to \mfr{i}_{b}^*$ and $\mfr{i}_b^!\to \mfr{q}_{b!}$ are isomorphisms. 
	Moreover, by \cite[Theorem IV.5.3]{FS21}, the maps $\ov{\mfr{q}}_{b*}\to \ov{\mfr{i}}_b^*$ and $\ov{\mfr{i}}_b^!\to \ov{\mfr{q}}_{b!}$ are also isomorphisms (see the proof of \cite[Proposition V.4.2]{FS21}). 
	Thus, it is enough to show that the maps $\ov{\mfr{i}}_b^*\ov{\mfr{p}}_b^!\to \mfr{i}_b^!\mfr{p}_b^*$ and $\mfr{i}_b^*\mfr{p}_b^!\to \ov{\mfr{i}}_b^!\ov{\mfr{p}}_b^*$ are isomorphisms. 
	Since $\ov{\mfr{p}}_b$ is $\ell$-cohomologically smooth, we have $\ov{\mfr{p}}_b^!\cong \ov{\mfr{p}}_b^*\otimes \ov{\mfr{p}}_b^!\Lambda$. Since $[\ast/M_b(E)]\cong \Bun_G^b\times_{\Bun_G} \cl{M}_b$, the same holds true for $\mfr{i}_b$. Thus, the claim follows from \cite[Proposition 23.12]{Sch17} and the projection formula. 

	The last thing to show is that $\Psi_{(G,P,b)}$ is excisive. For every point of $\Bun_M^{\preceq b}$, there is a parabolic subgroup $P'\subset P$ with a Levi subgroup $M'\subset M$ and $b'\in B(M')_\bas$ such that $\nu_{b'}$ is strictly $(M\cap P')$-codominant and $b'$ corresponds to the point of $\Bun_M^{\preceq b}$. We see from the case (1) that $\CT^\un_{M\cap \ov{P'},*}\vert_{\Bun_{M'}^{b'}}\to \CT^\un_{M\cap P'}\vert_{\Bun_{M'}^{b'}}$ is an isomorphism. Since $\CT^\un_{M\cap P'}\vert_{\Bun_{M'}^{b'}}=i^{b'*}\otimes K^{-1}_{\Bun_{M\cap P'}}$, the stalk of $\CT^\un_{\ov{P}*}\to \CT^\un_P$ at $b'$ is equal to $\CT^\un_{\ov{P'},*}\vert_{\Bun_{M'}^{b'}}\to \CT^\un_{P'}\vert_{\Bun_{M'}^{b'}}$ up to a twist. Since excision holds on $\Bun_M^{\preceq b}$ as in \cite[Proposition VII.7.3]{FS21}, $\Psi_{(G,P,b)}$ holds if $\Psi_{(G,P',b')}$ holds for all $(P',b')$. 
\end{proof}

The following corollary is speculated in \cite[Exercise 1.5.4]{Han24}. It will be used to show the compatibility of duality on the reducible part of the categorical correspondence. 

\begin{cor}
	For a torsion $\bb{Z}[\invp]$-algebra $\Lambda$, 
	there is a natural equivalence of functors $\Eis^\un_P\cong \bb{D}_\BZ\circ \Eis^\un_{\ov{P}} \circ \bb{D}_\BZ$ from $\cl{D}_\et(\Bun_M,\Lambda)^\omega$ to $\cl{D}_\et(\Bun_G,\Lambda)^\omega$. 
	Here, $\bb{D}_\BZ$ denotes the Bernstein--Zelevinsky involution. 
\end{cor}
\begin{proof}
	Thanks to \Cref{thm:unEis}, it is enough to construct a natural isomorphism $\Hom(\bb{D}_\BZ\circ \Eis^\un_P\circ \bb{D}_\BZ(A),B)\cong \Hom(A,\CT^\un_P(B))$ for $A \in \cl{D}_\et(\Bun_M,\Lambda)^\omega$ and $B\in \cl{D}_\et(\Bun_G,\Lambda)^\omega$. 
	It follows from that
	\begin{align*}
		\Hom(\bb{D}_\BZ\circ \Eis^\un_P\circ \bb{D}_\BZ(A),B) &\cong 
		\pi_{\Bun_G !}(\mfr{p}_{!}\mfr{q}^*\bb{D}_\BZ(A)\otimes B) \\
		&\cong \pi_{\Bun_P !}(\mfr{q}^*\bb{D}_\BZ(A)\otimes \mfr{p}^*B) \\
		&\cong \pi_{\Bun_M !}(\bb{D}_\BZ(A)\otimes \mfr{q}_{!}\mfr{p}^*B) \\
		&\cong \Hom(A,\CT^\un_P(B))
	\end{align*}
	where $\pi_{\Bun_{(-)}}$ denotes the structure map $\Bun_{(-)}\to \ast$. 
\end{proof}

Now, we move on to the case where $\Lambda$ is a $\Zl$-algebra and work with lisse-\'{e}tale sheaves. This $\ell$-adic coefficient system has a defect in that excision does not hold in general and a full six-functor formalism is not available. 
Nonetheless, we still have the following analogues in this coefficient system. 
We expect that the same strategy as in the torsion case should work in a more sophisticated $\ell$-adic coefficient system. 

Let $K_{\Bun_P}$ be the dualizing complex of $\Bun_P$ calculated in \cite{HI24}.
Though it a priori lives on $\Bun_P$, it can be defined over $\Bun_M$ and we regard $K_{\Bun_P}$ as a lisse-\'{e}tale sheaf on $\Bun_M$. 

\begin{defi}
	For a $\Zl$-algebra $\Lambda$, the unnormalized geometric constant term functor with coefficients in $\Lambda$ is defined as
	\[
		\CT^\un_P=\mfr{q}_\natural(\mfr{p}^*\otimes K_{\Bun_P}^{-1})\colon \cl{D}_\lis(\Bun_G,\Lambda)\to \cl{D}_\lis(\Bun_M,\Lambda). 
	\]
\end{defi}

\begin{prop}\label{prop:adjlis}
	For a $\Zl$-algebra $\Lambda$, $\CT^\un_P$ preserves small limits and admits a left adjoint. 
	The left adjoint of $\CT^\un_{\ov{P}}$ preserves compact objects and it is denoted by $\Eis^\un_P$. 
\end{prop}
\begin{proof}
	We apply \Cref{prop:ind} to the property $\Psi_{(G,P,b)}$ that $\CT^\un_P\vert_{\Bun_M^{\preceq b}}$ preserves small limits. 
	Since the restriction to each locally closed stratum $\Bun_G^b$ preserves small limits (\cite[Proposition VII.7.2]{FS21}), $\Psi_{(G,P,b)}$ satisfies excision. 
	Thus, as in the proof of \Cref{thm:unEis}, it is enough to show that $\mfr{q}_{b\natural}\mfr{p}_b^*$ and $\ov{\mfr{q}}_{b\natural} \ov{\mfr{p}}_b^*$ preserves small limits when $\nu_b$ is strictly $P$-codominant. 
	As in the proof of \cite[Proposition VII.7.1]{FS21}, $\mfr{q}_{b\natural}$ is an equivalence, so $\mfr{q}_{b\natural}\mfr{p}_b^*$ preserves small limits by \cite[Proposition VII.7.2]{FS21}. 
	Now, let $\CT_{\ov{P}}^{\un,b}=\ov{\mfr{q}}_{b\natural} (\ov{\mfr{p}}_b^*\otimes K_{\Bun_P}^{-1})$. We will show that $\CT_\ov{P}^{\un,b}$ is right adjoint to $i^b_!$. Here, $i^b_!$ is the functor of extension by zero at $b$ such that for every $A\in \cl{D}_\lis(\Bun_M^b,\Lambda)$, $i^b_!A$ is supported only at $b$ and $i^{b*}i^b_!A=A$. For safety, we fix a geometric point $\Spa(C,O_C)$ and change the base from $\ast$ to $\Spa(C,O_C)$. It does not matter by \cite[Proposition VII.7.3]{FS21}. We use the subscript $(-)_C$ to denote the base change to $\Spa(C,O_C)$. Moreover, for a compact open subgroup $K\subset M_b(E)$, the base change along $[\Spa(C,O_C)/K]\to [\Spa(C,O_C)/M_b(E)]$ is denoted by the subscript $(-)_{C,K}$. 

	Now, by \cite[Proposition 8.10]{Man22}, we may identify $\ov{\mfr{q}}_{b,C\natural}((-)\otimes K_{\Bun_P}^{-1})$ with the nuclear lower shriek functor $\ov{\mfr{q}}^\nuc_{b,C!}$.
	Then, we have $\CT_{\ov{P},C}^{\un,b}\circ i_{C!}^b=\ov{\mfr{q}}^\nuc_{b,C!}\ov{\mfr{p}}_{b,C}^*i_{C!}^b=\ov{\mfr{q}}^\nuc_{b,C!}\ov{\mfr{i}}_{b,C*}$. Here, excision holds on $\cl{M}_{b,C}$ with $\ov{\mfr{i}}_{b,C}$ and its complement, so we have $\ov{\mfr{i}}_{b,C*}=\cofib(j_!(\ov{\mfr{q}}_{b,C}\circ j)^*\to \ov{\mfr{q}}_{b,C}^*)$ with $j\colon \cl{M}_{b,C}^\circ = \cl{M}_{b,C}\setminus [\Spa(C,O_C)/M_b(E)]\hookrightarrow \cl{M}_{b,C}$ the complement of $\ov{\mfr{i}}_{b,C}$. Thus, $\ov{\mfr{q}}^\nuc_{b,C!}$ and $\ov{\mfr{i}}_{b,C*}$ are both preserved under base change along $[\Spa(C,O_C)/K]\to [\Spa(C,O_C)/M_b(E)]$. Now, we have $\ov{\mfr{q}}^\nuc_{b,C,K!}\ov{\mfr{i}}_{b,C,K*}=\ov{\mfr{q}}^\nuc_{b,C,K*}\ov{\mfr{i}}_{b,C,K*}=\ov{\mfr{q}}^\lis_{b,C,K*}\ov{\mfr{i}}_{b,C,K*}=\id$ for every $K$ since $[\Spa(C,O_C)/K]$ has a quasicompact open neighborhood in $\cl{M}_{b,C,K}$. Since this construction is natural in $K$, we see that $\CT_{\ov{P},C}^{\un,b}\circ i_{C!}^b=\ov{\mfr{q}}^\nuc_{b,C!}\ov{\mfr{i}}_{b,C*}\cong \id$. 

	Now, we obtain a map $\Hom(i^b_{C!}A,B)\to \Hom(A,\CT^{\un,b}_{\ov{P},C}B)$ for $A\in \cl{D}_\lis(\Bun_{M,C}^b,\Lambda)$ and $B\in \cl{D}_\lis(\Bun_{G,C},\Lambda)$. Since we have $\Hom(i^b_{C!}A,B)\cong \Hom(\ov{\mfr{i}}_{b,C*}A,\ov{\mfr{p}}_{b,C}^*B)$, we will show that $\Hom(\ov{\mfr{i}}_{b,C*}A,B)\cong \Hom(A,\ov{\mfr{q}}^\nuc_{b,C!}B)$ for $B\in \cl{D}_\lis(\cl{M}_{b,C},\Lambda)$. We will deduce this from $\Hom(\ov{\mfr{i}}_{b,C,K*}A,B) \cong \Hom(A,\ov{\mfr{q}}^\nuc_{b,C,K!}B)$ for every open pro-$p$ subgroup $K\subset M_b(E)$. Now, $\ov{\mfr{i}}_{b,C,K*}$ is left adjoint to $\fib(\ov{\mfr{q}}_{b,C,K*}\to (\ov{\mfr{q}}_{b,C,K}\circ j)_*j^*)$. Both $\ov{\mfr{q}}_{b,C,K*}$ and $(\ov{\mfr{q}}_{b,C,K}\circ j)_*$ preserve colimits by \cite[Proposition VII.7.2]{FS21} and \cite[Lemma VII.7.5]{FS21}, so $\ov{\mfr{i}}_{b,C,K*}$ preserves compact objects. Thus, we may assume that $A$ is compact and $\Lambda=\Zl$. Then, since $\ov{\mfr{i}}_{b,C,K*}A$ is compact and $\ov{\mfr{q}}^\nuc_{b,C,K!}$ preserves colimits, we may assume that $B=f_\natural \Zl$ for some quasicompact separated $\ell$-cohomologically smooth map $f\colon X\to \cl{M}_{b,C,K}$ factoring through a quasicompact open subspace of $\cl{M}_{b,C,K}$. By \cite[Lemma 5.2 (iv)]{Man22}, $g_\natural$ preserves $\ell$-adically complete objects if $g$ is quasicompact, separated and $\ell$-cohomologically smooth, so $f_\natural \Zl$ and $(\ov{\mfr{q}}_{b,C,K}\circ f)_\natural \Zl\otimes K_{\Bun_P}^{-1}$ are both $\ell$-adically complete. Thus, the claim can be reduced to the torsion coefficients and it follows from the fact that $\ov{\mfr{i}}_b^!\cong \ov{\mfr{q}}_{b!}$ in the torsion case. Finally, the above comparison is natural in $K$, so we can descend along $[\Spa(C,O_C)/K]\to [\Spa(C,O_C)/M_b(E)]$ to obtain $\Hom(\ov{\mfr{i}}_{b,C*}A,B)\cong \Hom(A,\ov{\mfr{q}}^\nuc_{b,C!}B)$. 
\end{proof}

We verify that this geometric Eisenstein series has the expected description in the following basic cases. 

\begin{lem}\label{lem:explicit}
	Let $b\in B(M)$ be a basic $\sigma$-conjugacy class and let $A\in \cl{D}_\lis(\Bun_M^b,\Lambda)$. 
	\begin{enumerate}
		\item If $\nu_b$ is strictly $P$-codominant, we have $\Eis^\un_P(A)\cong i^b_!A$. 
		\item If $b$ is basic in $G$, $\Eis^\un_P(A)$ lies in $\cl{D}_\lis(\Bun_G^b, \Lambda) \cong \cl{D}(G_b(E),\Lambda)$ and corresponds to the unnormalized parabolic induction $\Ind^{G_b}_{P_b}(A)$. 
		\item If $\nu_b$ is strictly $P$-dominant, we have $\Eis^\un_P(A)\cong \mfr{p}_{b\natural }(\mfr{q}_b^*A\otimes K_{\Bun_P}^{-1})$. 
	\end{enumerate}
\end{lem}
\begin{proof}
	The case (1) is treated in the proof of \Cref{prop:adjlis}. 
	The case (2) follows from the argument in \Cref{lem:repsec}. In case (3), we have $\CT_\ov{P}^\un=i^{b*}\otimes K_{\Bun_\ov{P}}^{-1}$. Since $K_{\Bun_\ov{P}}^{-1}=K_{\Bun_P}$, we obtain the description of its left adjoint from \cite[Proposition VII.7.2]{FS21}.
\end{proof}

We normalize geometric Eisenstein series and geometric constant terms as follows so that they are compatible with spectral counterparts. 

\begin{defi}\label{defi:normalization}
	Let $\Lambda$ be a torsion $\bb{Z}[\invp]$-algebra or a $\Zl$-algebra with a square root $\sqrt{q}\in \Lambda$. 
	The normalized geometric Eisenstein series functor and the normalized geometric constant term functor are defined as
	\[
		\Eis_P=\Eis^\un_P((-)\otimes K_{\Bun_P}^{1/2}),\hspace{5pt}
		\CT_P=\CT^\un_P(-)\otimes K_{\Bun_P}^{1/2}.
	\]
\end{defi}

It is easy to see that the second adjointness still holds after taking this normalization. This is compatible with the normalization in \cite{HI24}. As in \cite[Lemma 4.17]{HI24}, normalized geometric Eisenstein series and normalized geometric constant terms are associative. 

\subsection{Relation to the classical approach in the geometric Langlands}\label{ssc:DG}

Our proof of \Cref{thm:unEis} decomposes the correspondence along $\Bun_P$ into certain basic correspondences and finally relies on the representation-theoretic second adjointness proved in \cite{DHKM24} and hyperbolic localization for diamonds established in \cite{FS21}. 
On the contrary, in the geometric Langlands, the second adjointness is proved uniformly via hyperbolic localization in \cite{DG16}. 
In this section, we explain the difference between the local and geometric settings. 

An approach taken in \cite{DG16} is a direct construction of an adjunction between $\Eis^\un_P$ and $\CT^\un_{\ov{P}}$. 
More precisely, they constructed (actually using $\Eis^\un_{P*}$ and $\CT^\un_{\ov{P}*}$) a unit $\alpha\colon \id\to \CT^\un_{\ov{P}}\circ \Eis^\un_P$ and a counit $\beta\colon \Eis^\un_P\circ \CT^\un_{\ov{P}}\to \id$ equipped with certain compatibility. 

The unit $\alpha$ is constructed from a morphism of correspondences induced by the open immersion $\Bun_M\hookrightarrow \Bun_\ov{P}\times_{\Bun_G} \Bun_P$. 
On the other hand, the counit $\beta$ is constructed via a \textit{degeneration} of correspondences from $\Bun_G$ to $\Bun_P\times_{\Bun_M}\Bun_{\ov{P}}$. 
This degeneration is constructed from the wonderful degeneration $\wtd{G}$ over $\Al$ equipped with an equivariant $\Gm$-action. 
The family $\wtd{G}\vert_\Gm$ is isomorphic to a constant family $G\times \Gm$ and the fiber $\wtd{G}_0$ is isomorphic to $P\times_M \ov{P}$. 
The stack $\Bun_\wtd{G}$ plays the role of a degeneration of $\Bun_G$ to $\Bun_P\times_{\Bun_M} \Bun_\ov{P}$. 

Our observation is that this degeneration $\Bun_\wtd{G}$ behaves differently in the geometric and local settings. 
Let $\Hom(X,\Al)$ (resp.\ $\Hom(X,\Gm)$) denote the functor parametrizing the set of morphisms from the relative curve $X$ to $\Al$ (resp.\ $\Gm$). 
The stack $\Bun_\wtd{G}$ is a family over $\Hom(X,\Al)$ which is constant over $\Hom(X,\Gm)$. 
In the geometric setting, $\Hom(X,\Al)$ (resp.\ $\Hom(X,\Gm)$) is isomorphic to $\Al$ (resp.\ $\Gm$), so we may apply hyperbolic localization to the family $\Bun_\wtd{G}$, which enables us to obtain the counit $\beta$. 
However, in the local setting, $\Hom(X,\Al)$ (resp.\ $\Hom(X,\Gm)$) is isomorphic to $\underline{E}$ (resp.\ $\underline{E}^\times$). 
Thus, the family $\Bun_\wtd{G}$ is totally disconnected in a sense, and hyperbolic localization cannot be applied as it is.
It is suggested from the representation-theoretic behavior studied in \cite{BK15} (with complex coefficients) that we need to study the asymptotic behavior of the degeneration. 
We think that this difference stems from the essential difference between schemes and $\bb{Q}_p$-manifolds. 

\section{Cuspidal supports}
\label{sec:spact}

\subsection{Spectral Eisenstein series}\label{ssc:specEis}

In this section, we review the properties of the moduli stacks of $L$-parameters developed in \cite{DHKM20}, \cite{Zhu21} and \cite{FS21}. 
Here, we assume $\Lambda=\Qla$. 

For a linear algebraic group $H$ over $E$ with an $L$-group ${}^LH$ over $\Qla$, the derived moduli stack of $L$-parameters valued in ${}^LH$ is denoted by $\Par_{{}^LH}$. 
It is quasi-smooth with trivial dualizing complex (see \cite[Remark 2.3.8]{Zhu21}) and it is classical if $H$ is reductive. 
From now on, let $G$ be a quasi-split connected reductive group over $E$. 
For a parabolic subgroup $P$ of $G$ with a Levi subgroup $M$, the spectral Eisenstein series is the functor associated with the following correspondence. 
\[
\begin{tikzcd}
    & \Par_\LP \arrow[ld, "p^\spec"']  \arrow[rd, "q^\spec"] & \\
    \Par_\LG & & \Par_\LM
\end{tikzcd}
\]
Here, $p^\spec$ is proper and schematic, and $q^\spec$ is quasi-smooth. 
The key player on the spectral side is the derived category of ind-coherent sheaves on $\Par_\LG$ denoted by $\IndCoh(\Par_\LG)$. 
Note that the nilpotent condition on singular supports automatically holds in characteristic $0$ (\cite[Proposition VIII.2.11]{FS21}). 
Let $\Sing(\Par_\LG)$ be the stack of singularities of $\Par_\LG$. 
Its closed points can be described as follows (see \cite[Corollary VIII.2.3]{FS21}).
\[\Sing(\Par_\LG)(\Qla)=\{\varphi\colon W_E\to \LG(\Qla), \xi \in H^0(W_E,\widehat{\mfr{g}}^\ast_\varphi(1)) \}/\sim \]
Here, the action of $W_E$ on the dual $\widehat{\mfr{g}}^\ast$ of the Lie algebra of $\widehat{G}$ is twisted by $\varphi$ and a Tate twist and the equivalence relation $\sim$ denotes the conjugation by $\widehat{G}$. 
From this description, we see that the singularities of $\Par_\LG$ are nilpotent over $\Qla$.
We define the spectral Eisenstein series as the functor 
\[\Eis^\spec_P=(p^\spec)_*(q^\spec)^!\colon \IndCoh(\Par_\LM)\to \IndCoh(\Par_\LG).\]
In this definition, we use the formalism developed in \cite{GR17I} and \cite{GR17II}. 
The functors $(p^\spec)_*$ and $(q^\spec)^!$ preserve colimits and should be regarded as spectral counterparts of $\mfr{p}_!$ and $\mfr{q}^*$.
The functor $\Eis_P^\spec$ factors through $\IndCoh(\Par_\LP)$ and we can exactly determine the subcategory of $\IndCoh(\Par_\LP)$ it factors through in the same way as \cite[Proposition 13.4.4]{AG15}. 
Let $\widehat{\mfr{p}}$ be the Lie algebra of $\widehat{P}$ and $\widehat{\mfr{u}}$ be its unipotent radical. 

\begin{prop}\label{prop:nilpp}
	Consider the morphism $\Sing(q^\spec)\colon \Sing(\Par_\LM)_{\Par_\LP}\to \Sing(\Par_\LP)$ of the stacks of singularities and let $\Nilp_P\subset \Sing(\Par_\LP)$ be the image of the morphism. Its closed points can be described as 
	\[\Nilp_P(\Qla)=\{\varphi\colon W_E\to \LP(\Qla), \xi \in H^0(W_E,\widehat{\mfr{p}}^\ast_\varphi(1)) \st \xi\vert_{\widehat{\mfr{u}}}=0\}/\sim \]
	and the essential image of $(q^\spec)^!$ generates the subcategory $\IndCoh_{\Nilp_P}(\Par_\LP)$. 
\end{prop}
\begin{proof}
	The first claim follows from the explicit descriptions of the stacks of singularities. 
	The second claim can be proved in the same way as \cite[Proposition 13.4.4]{AG15}. 
	It is because for each open subgroup $P$ of the wild inertia subgroup of $W_E$, the stack of $L$-parameters trivial on $P$ is represented by a quotient of an affine derived scheme, and $q^\spec$ is represented by a quotient of an affine morphism by a unipotent group as in the geometric situation (see \cite[Proposition 2.2.3]{Zhu21}). 
\end{proof}

The coarse moduli space of $\Par_{\LG}$ is denoted by $\Par^\ssi_{\LG}$. 
Its closed point corresponds to a $\widehat{G}$-conjugacy class of semisimple $L$-parameters. 
For each semisimple $L$-parameter $\varphi\colon W_E\to \LG(\Qla)$, there is a Levi subgroup $M$ of $G$ such that $\varphi$ is $\widehat{G}$-conjugate to a supercuspidal $L$-parameter of $M$ and the Levi subgroup $M$ is unique up to $G$-conjugacy (see \cite[Section 2.3]{BMO22}). 
We call the conjugacy class of $M$ the cuspidal support of the semisimple $L$-parameter $\varphi$. 
As we show in the following proposition, cuspidal supports are locally constant on $\Par^\ssi_{\LG}$. 

\begin{prop}
	The cuspidal supports of the closed points of $\Par_\LG^\ssi$ are constant on each connected component of $\Par^\ssi_\LG$. 
\end{prop}
\begin{proof}
	Let $\varphi\colon W_E\to \LG(\Qla)$ be a semisimple $L$-parameter. 
	We may assume that $\varphi$ comes from a supercuspidal $L$-parameter of a Levi subgroup $M$ of $G$.
	Since the torus $Z(\widehat{M})^{\Gamma,\circ}$ is contained in the centralizer $C_{\varphi(I_F)}$ of $\varphi(I_F)$, we can take a maximal torus $\widehat{T}_\varphi$ of $C_{\varphi(I_F)}$ inside $\widehat{M}$. 
	Then, the claim follows from the description of the connected component containing $\varphi$ in \cite[Theorem 1.7]{DHKM20} which states that every closed point of it corresponds to an unramified twist of $\varphi$ by some $\hat{t}\in \widehat{T}_\varphi$. 
\end{proof}
\begin{defi}
	Let $[M]$ be a conjugacy class of Levi subgroups of $G$. 
	The closed and open subscheme of $\Par^\ssi_\LG$ consisting of semisimple $L$-parameters with cuspidal supports $[M]$ is denoted by $\Par^\ssi_{\LG,[M]}$. 
\end{defi}

The subscheme $\Par^\ssi_{\LG,[M]}$ is also described as the image of the map $\Par^\ssi_{{}^LM,[M]}\to \Par^\ssi_{\LG}$. 
The inverse image of $\Par^\ssi_{\LG,[M]}$ along the map $\Par_\LG\to \Par^\ssi_\LG$ is denoted by $\Par_{\LG,[M]}$. 
We have a decomposition $\Par_\LG=\coprod_{[M]} \Par_{\LG,[M]}$. It can be used to determine the subcategory of $\IndCoh(\Par_\LG)$ generated by the essential images of the spectral Eisenstein series $\Eis^\spec_P$ as in \cite[Theorem 13.4.2]{AG15}. 

\begin{prop}\label{prop:esssurj}
	The category $\IndCoh(\Par_{\LG,[M]})$ is generated by the essential images of the functors 
	\[\Eis^\spec_{P,[M]}\colon \IndCoh(\Par_{{}^LM,[M]})\to \IndCoh(\Par_{\LG,[M]})\]
	for all parabolic subgroups $P$ with Levi subgroups $M$. 
\end{prop}
\begin{proof}
	By \Cref{prop:nilpp}, it is enough to show that the essential images of 
	\[(p^\spec)_*\colon \IndCoh_{\Nilp_P}(\Par_{\LP,[M]})\to \IndCoh(\Par_{\LG,[M]})\]
	for all $P$ generate the target category. 
	Since $p^\spec$ is proper and schematic, it is enough to show that $\coprod_{P}\Sing(p^\spec)^{-1}(\Nilp_P)\to \Sing(\Par_{\LG,[M]})$ is surjective on $\Qla$-points as in \cite[Theorem 13.4.2]{AG15}. 
	Take a closed point $(\varphi,\xi)\in \Sing(\Par_\LG)(\Qla)$ via the explicit description and fix an isomorphism $\widehat{\mfr{g}}^\ast\cong \widehat{\mfr{g}}$ as $\widehat{G}$-representations (coming from the Killing form). 
	Since $\widehat{\mfr{u}}^{\perp}\cong\widehat{\mfr{p}}$ via this isomorphism, it is enough to show that after taking a $\widehat{G}$-conjugation, $\varphi$ factors through $\LP$ and $\xi$ lies in $\widehat{\mfr{p}}$ for some $P$. 

	Let $\varphi^\ssi$ be the semisimplification of $\varphi$ and let $N$ be the monodromy of $\varphi$. 
	Fix a geometric Frobenius $\Fr\in W_E$ and let $u=\varphi(\Fr)\varphi^\ssi(\Fr)^{-1}$. 
	The elements $u$, $N$ and $\xi$ are unipotent and commute with each other. 
	The centralizer $C_{\varphi^\ssi(I_F)}$ is reductive and normalized by $\varphi^\ssi(\Fr)$, and it contains $u$, $N$, $\xi$ and some power of $\varphi^\ssi(\Fr)$ since the image of $\varphi^\ssi(I_F)$ is finite. 
	Moreover, we have $\Ad_{\varphi^\ssi(\Fr)}(u)=u$, $\Ad_{\varphi^\ssi(\Fr)}(N)=q^{-1}N$ and $\Ad_{\varphi^\ssi(\Fr)}(\xi)=q\xi$. 
	Thus, we may apply \Cref{lem:Jordan} to the reductive group generated by $C_{\varphi^\ssi(I_F)}$ and $\varphi^\ssi(\Fr)$ for the elements $\varphi^\ssi(\Fr)$, $u$, $\exp(N)$ and $\exp(\xi)$. 
	Then, we get a cocharacter $\mu\colon \Gm\to \widehat{G}$ commuting with $\varphi^\ssi$ such that $u$ (resp.\ $N$ and $\xi$) lies in $\widehat{G}^{\mu>0}$ (resp.\ $\widehat{\mfr{g}}^{\mu>0}$). 
	The cocharacter $\mu$ takes values in $S_{\varphi^\ssi}^\circ$, so by replacing $\varphi^\ssi$ so that it takes values in $\LM$ and then taking a conjugation by an element of $S_{\varphi^\ssi}$, we may assume by \cite[Lemma 2.5]{BMO22} that $\mu$ is a cocharacter of $Z(\widehat{M})^{\Gamma,\circ}$. 
	Then, for every $P$ such that $\widehat{P}$ contains $\widehat{G}^{\mu>0}$, 
	we see that $u$ (resp.\ $N$ and $\xi$) lies in $\LP$ (resp.\ $\widehat{\mfr{p}}$). 
\end{proof}

\begin{lem}\label{lem:Jordan}
	Let $G$ be a (possibly disconnected) reductive group over an algebraically closed field $k$ of characteristic $0$. 
	Let $s$ be a semisimple element of $G(k)$ and let $u_1,\ldots,u_n$ be unipotent elements of $G(k)$ satisfying
	\begin{enumerate}
		\item $su_is^{-1}=u_i^{q_i}$ for some $q_i\in k^\times$, and 
		\item $u_iu_j=u_ju_i$ for every $1\leq i,j \leq n$. 
	\end{enumerate}
	Here, $u_i^{q_i}$ denotes the image of $q_i$ along the exponential map $\exp_{u_i}\colon \Ga \to G$ sending $1$ to $u_i$. 
	Then, there is a cocharacter $\mu\colon \Gm\to G$ satisfying $s\in G^{\mu=0}$ and $u_i\in G^{\mu>0}$ for every $1\leq i \leq n$. 
\end{lem}
\begin{proof}
	We prove by induction on $n$. Since $s$ is semisimple, it is contained in a maximal torus of $G$, so the claim follows for $n=0$. 
	Suppose that $n>0$. 
	As in \cite[Lemma 1.12]{Imai21}, there is a homomorphism $\theta\colon \SL_2\to G$ such that $\theta\left(\left(\begin{matrix} 1 & 1\\0 & 1\end{matrix}\right)\right)=u_1$ and $s'=s\cdot (\theta\circ \alpha^\vee)(q_1^{-1/2})$ commutes with $\theta$ where $\alpha^\vee$ is the diagonal cocharacter of $\SL_2$. 
	By \cite[Proposition 2.4]{BV85}, the centralizer of $u_1$ has a Levi decomposition $G^\theta\ltimes U^{u_1}$. 
	Here, $G^\theta$ is the centralizer of $\theta$ and the unipotent radical $U^{u_1}$ is positive with respect to the cocharacter $\theta\circ \alpha^\vee$. 
	Take the corresponding decomposition $u_i=u'_i\cdot v'_i$ for each $i$. 
	Since $\theta\circ \alpha^\vee$ normalizes the centralizer of $u_1$ and commutes with $G^\theta$, we have $s'u'_is'^{-1}=u'^{q_i}_i$.
	Thus, we may apply the induction hypothesis to $s'$ and $u'_2,\ldots,u'_n$ in $G^\theta$. 
	Then, we get a cocharacter $\mu'\colon \Gm\to G^\theta$ which commutes with $s'$ such that $u'_2,\ldots,u'_n$ lie in the positive unipotent radical with respect to $\mu'$. 
	Consider the cocharacter $\mu'_N=\mu'(\theta\circ \alpha^\vee)^N$ for a positive integer $N$. 
	Since $\mu'$ commutes with $s'$ and $\theta$, $\mu'_N$ commutes with $s$. 
	Moreover, if $N$ is sufficiently large, $u_1,\ldots,u_n$ lie in the positive unipotent radical with respect to $\mu'_N$, so the claim holds for $\mu'_N$ for sufficiently large $N$. 
\end{proof}

\subsection{Compatibility with geometric Eisenstein series}
In this section, we first let $\Lambda$ be a $\Zl[\sqrt{q}]$-algebra such that $\pi_0(Z)$ is invertible in $\Lambda$. Here, $Z$ is the center of $G$.  
The ring of global sections on $\Par_\LG$ is denoted by $\cl{Z}^\spec(G,\Lambda)$ and this is what we call the spectral Bernstein center. 
The action of $\cl{Z}^\spec(G,\Lambda)$ on the identity endofunctor of $\cl{D}_\lis(\Bun_G,\Lambda)$ is constructed in \cite{FS21} and it is expected to correspond to the multiplicative action on the identity endofunctor of $\IndCoh(\Par_\LG)$. Thus, as claimed in \cite[Exercise 1.5.5]{Han24}, the action of the spectral Bernstein center is expected to commute with geometric Eisenstein series. 
In this section, we verify this compatibility \textit{up to filtration} using \Cref{prop:ind}. 
The category of lisse-\'{e}tale sheaves on $\Bun_G$ with quasicompact supports is denoted by $\cl{D}_\lis(\Bun_G,\Lambda)^\qc$. 
For an element $f\in \cl{Z}^\spec(G,\Lambda)$, its restriction to $\cl{Z}^\spec(M,\Lambda)$ is denoted by $f^M$. 

\begin{thm}\label{thm:compat}
	For every $A\in \cl{D}_\lis(\Bun_M,\Lambda)^\qc$, there is a finite filtration $\Fil^\bullet\Eis_P(A)$ of $\Eis_P(A)$ with an action of $\cl{Z}^\spec(M,\Lambda)$ extending the spectral action on $\Eis_P(A)$ such that for every $f\in \cl{Z}^\spec(G,\Lambda)$, the action of $f^M$ on each subquotient $\gr^{\bullet}(\Eis_P(A))$ is equal to the spectral action of $f$.  
\end{thm}

We need to keep track of twists carefully to see that the normalization in \Cref{defi:normalization} is the correct one. 
For this, we explicitly determine the twisted action of the spectral Bernstein center. We follow the argument noted at the end of \cite[IX.6.4]{FS21}. 

For a smooth character $\xi\colon G(E)\to \Lambda^\times$, the endomorphism of the Bernstein center $\cl{Z}(G(E),\Lambda)=\pi_0\End(\id_{\cl{D}(G(E),\Lambda)})$ sending $F$ to $F((-)\otimes \xi)\otimes \id_{\xi^{-1}}$ is denoted by $\tau_\xi$. It means that for every smooth representation $\pi$ of $G(E)$ over $\Lambda$ and every $F\in \cl{Z}(G(E),\Lambda)$, the action of $\tau_\xi(F)$ on $\pi$ is given by $F(\pi\otimes \xi)$. 
For an $L$-parameter $\varphi\colon W_E\to {}^LG(\Lambda)$ whose projection to $\widehat{G}(\Lambda)$ factors through $Z(\widehat{G})^\Gamma(\Lambda)$, there is an endomorphism of $\Par_\LG$ which sends an $L$-parameter $\psi$ to $\varphi\cdot \psi$ as $1$-cocycles valued in $\widehat{G}(\Lambda)$. The endomorphism of $\cl{Z}^\spec(G,\Lambda)$ obtained by the pullback of this endomorphism is denoted by $\tau_\varphi$. 

\begin{lem}\label{lem:twist}
	For a character $\alpha\colon G\to \Gm$ and a smooth character $\xi\colon E^\times\to \Lambda^\times$, let $\widehat{\alpha}\colon \Gm\to \widehat{G}$ be the dual cocharacter of $\alpha$ and $\varphi_\xi\colon W_E\to \Lambda^\times$ be the character corresponding to $\xi$ via the local class field theory. 
	Then, the following diagram commutes. 

	\[
	\begin{tikzcd}
		\cl{Z}^\spec(G,\Lambda) \arrow[d,"\tau_{\widehat{\alpha}(\Lambda)\circ \varphi_\xi}"] \arrow[r] & \cl{Z}(G(E),\Lambda) \arrow[d,"\tau_{\xi\circ \alpha(E)}"] \\
		\cl{Z}^\spec(G,\Lambda) \arrow[r] & \cl{Z}(G(E),\Lambda) 
	\end{tikzcd}
	\]
\end{lem}
\begin{proof}
	Let $f$ be an element of $\cl{Z}^\spec(G,\Lambda)$ and $\pi$ be a smooth $G(E)$-representation over $\Lambda$. 
	By applying \cite[Theorem IX.6.1]{FS21} to the homomorphism $\id\times \alpha\colon G\to G\times \Gm$, we get $f\vert_\pi=(\id\times \widehat{\alpha})^*(f)\vert_{\pi(-\xi\circ \alpha(E))\otimes \xi}$. 
	By \cite[Proposition IX.6.2]{FS21} and \cite[Proposition IX.6.5]{FS21}, the action on $\pi(-\xi\circ \alpha(E))\otimes \xi$ of $\cl{Z}^\spec(G\times \Gm,\Lambda)$ factors through $\ev_{\varphi_\xi}\colon \cl{Z}^\spec(G\times \Gm, \Lambda)\to \cl{Z}^\spec(G,\Lambda)$. 
	Since the composition of $\ev_{\varphi_\xi}$ with $(\id\times \widehat{\alpha})^*$ is the pullback along the multiplication by $\widehat{\alpha}(\Lambda)\circ \varphi_\xi$, we have $f\vert_\pi=\tau_{\widehat{\alpha}(\Lambda)\circ \varphi_\xi}(f)\vert_{\pi(-\xi\circ \alpha(E))}\otimes \id_{\xi\circ\alpha(E)}$. 
\end{proof}

\begin{proof}[Proof of \Cref{thm:compat}]
	We apply \Cref{prop:ind} to the property $\Psi_{(G,P,b)}$ that the claim holds for every $A\in \cl{D}_\lis(\Bun_M^{\preceq b},\Lambda)$. 
	First, suppose that $b$ is basic and $\nu_b$ is strictly $P$-codominant. 
	In this case, we have $\CT_P\vert_{\Bun_M^b}=i^{b*}\otimes K_{\Bun_P}^{-1/2}$ and it follows from \cite[Theorem IX.7.2]{FS21} that $\CT_P\vert_{\Bun_M^b}(f)=\tau_{\widehat{\alpha}(\Lambda)\circ \varphi_\xi}(f^M)\otimes \id_{K_{\Bun_P}^{-1/2}}$ with $\alpha=2\rho_G-2\rho_M$ and $\xi$ the unramified character sending a uniformizer to $\sqrt{q}$. 
	Here, $\rho_G$ and $\rho_M$ are the half-sums of positive roots and $\alpha$ is regarded as a character of $M_b$. 
	By the calculation of \cite{HI24}, $K_{\Bun_P}^{-1/2}$ is equal to $\xi\circ\alpha(E)$ on $\Bun_M^b$. 
	Thus, we have $\CT_P\vert_{\Bun_M^b}(f)=f^M$. 
	In particular, since $\Eis_{\ov{P}}\vert_{\Bun_M^b}$ is left adjoint to $\CT_P\vert_{\Bun_M^b}$, we have $\Eis_{\ov{P}}\vert_{\Bun_M^b}(f^M)=f$. 
	Moreover, it follows from \Cref{lem:explicit} that the image of $\Eis_P\vert_{\Bun_M^b}$ is supported only at $b$ and $\CT_P\vert_{\Bun_M^b} \circ \Eis_P\vert_{\Bun_M^b} = \id_{\cl{D}_\lis(\Bun_M^b, \Lambda)}$, so we have $\Eis_P\vert_{\Bun_M^b}(f^M)=f$. 
	Thus, $\Psi_{(G,P,b)}$ holds when $b$ is basic and $\nu_b$ is strictly $P$-codominant or strictly $P$-dominant. 
	Moreover, when $b$ is basic in $G$, $\Psi_{(G,P,b)}$ follows from \cite[Corollary IX.7.3]{FS21} by handling the character twists in the same way. 
	Thus, it is enough to show that $\Psi_{(G,P,b)}$ is excisive and associative to apply \Cref{prop:ind}. 
	
	First, we show that $\Psi_{(G,P,b)}$ is excisive. Take the excision filtration of the identity functor on $\cl{D}_\lis(\Bun_M^{\preceq b},\Lambda)$ as in \cite[Proposition VII.7.3]{FS21}. Each subquotient is written as $\Eis_{M\cap P'}\vert_{\Bun_{M'}^{b'}}\circ \CT_{M \cap P'}\vert_{\Bun_{M'}^{b'}}$ for a parabolic subgroup $P'\subset P$ with a Levi subgroup $M'\subset M$ and $b'\in B(M')_\bas$ such that $\nu_{b'}$ is strictly $(M\cap P')$-codominant. 
	By applying this filtration to $A$, we get a finite filtration of $\Eis_P(A)$ each of whose subquotients is written as $\Eis_{P'}\vert_{\Bun_{M'}^{b'}}\circ \CT_{M\cap P'}\vert_{\Bun_{M'}^{b'}}(A)$. The action of $f^M$ on each subquotient is given by $\Eis_{P'}\vert_{\Bun_{M'}^{b'}}(f^{M'})$. Thus, if we have $\Psi_{(G,P',b')}$ for all $(P',b')$, we get a finite filtration of each subquotient and obtain the claim by combining these filtrations. 
	It remains to show that $\Psi_{(G,P,b)}$ is associative. Suppose that $\Psi_{(G,P',b)}$ and $\Psi_{(M',M'\cap P,b)}$ hold for some parabolic subgroup $P'\supset P$ with a Levi subgroup $M'\supset M$. Then, we may suppose that $\Fil^\bullet\Eis_{M'\cap P}(A)$ is supported on $\Bun_{M'}^{\preceq b}$. Then, we may apply $\Psi_{(G,P',b)}$ to each subquotient $\gr^\bullet\Eis_{M'\cap P}(A)$ to obtain a finite filtration of $\Eis_P(A)$. 
\end{proof}

As a corollary, the action of idempotents in the spectral Bernstein center can be shown to commute with geometric Eisenstein series. 

\begin{cor}\label{cor:idemcompat}
	For every $A\in \cl{D}_\lis(\Bun_M,\Lambda)$ and every idempotent $e\in \cl{Z}^\spec(G,\Lambda)$, we have $\Eis_P(e^M\vert_A)=e\vert_{\Eis_P(A)}$.
\end{cor}
\begin{proof}
	Since $\Eis_P$ preserves colimits and $\cl{D}_\lis(\Bun_M,\Lambda)$ is compactly generated, we may assume that $A$ is compact. 
	Then, $A$ is supported on a quasicompact open substack of $\Bun_M$, so we may apply \Cref{thm:compat} to $A$ and obtain a finite (decreasing) filtration $\Fil^\bullet \Eis_P(A)$. 
	We will show that $\Fil^\bullet \Eis_P(e^M)=e\vert_{\Fil^\bullet \Eis_P(A)}$ by induction. 
	Suppose that the equality holds for $\Fil^{n+1}\Eis_P(A)$. 
	The difference $d=\Fil^n\Eis_P(e^M)-e\vert_{\Fil^n\Eis_P(A)}$ is zero on $\Fil^{n+1}\Eis_P(A)$, so it factors through $\gr^n\Eis_P(A)$. 
	Moreover, since $\gr^n\Eis_P(e^M)=e\vert_{\gr^n\Eis_P(A)}$, it also factors through $\Fil^{n+1}\Eis_P(A)$. 
	Thus, we have $d^2=0$. 
	Since $e$ is idempotent, we have $e^2=e$ and $(d+e)^2=d+e$ in $\End(\Fil^n\Eis_P(A))$. 
	Since $e$ lies in the center of $\End(\Fil^n\Eis_P(A))$, we have $d=2de$. 
	By multiplying $e$, we have $de=0$ and so $d=0$. 
	Thus, we have the equality $\Fil^n \Eis_P(e^M)=e\vert_{\Fil^n \Eis_P(A)}$. 
\end{proof}

We apply this to the decomposition with respect to cuspidal supports. 
From now on, assume that $\Lambda=\Qla$ as in \Cref{ssc:specEis}. 
For each conjugacy class $[M]$ of Levi subgroups of $G$, let $e_{[M]}$ be the idempotent in $\cl{Z}^\spec(G,\Qla)$ corresponding to the closed and open substack $\Par_{\LG,[M]}$ of $\Par_\LG$. The subcategory of $\cl{D}_\lis(\Bun_G,\Qla)$ fixed by $e_{[M]}$ is denoted by $\cl{D}_\lis(\Bun_G,\Qla)_{[M]}$. 
It induces a decomposition 
$\cl{D}_\lis(\Bun_G,\Qla) = \bigoplus_{[M]} \cl{D}_\lis(\Bun_G,\Qla)_{[M]}$. 
First, \Cref{cor:idemcompat} implies the invariance of cuspidal supports along geometric Eisenstein series. 

\begin{prop}
	The functor $\Eis_P$ sends $\cl{D}_\lis(\Bun_M,\Qla)_{[M]}$ into $\cl{D}_\lis(\Bun_G,\Qla)_{[M]}$. 
\end{prop}

The induced functor from $\cl{D}_\lis(\Bun_M,\Qla)_{[M]}$ to $\cl{D}_\lis(\Bun_G,\Qla)_{[M]}$ is denoted by $\Eis_{P,[M]}$. 
As shown in \Cref{prop:esssurj}, the spectral counterparts $\Eis^\spec_{P,[M]}$ for all $P$ generate the target category $\IndCoh(\Par_{\LG,[M]})$. 
Thus, the compatibility of Eisenstein series predicts that the functors $\Eis_{P,[M]}$ for all $P$ also generate the target category $\cl{D}_\lis(\Bun_G,\Qla)_{[M]}$. 
This expectation can be phrased differently in terms of cuspidality. 

\begin{defi}
	The subcategory of $\cl{D}_\lis(\Bun_G,\Qla)$ consisting of the objects $A$ such that $\CT_P(A)=0$ for every proper parabolic subgroup $P$ of $G$ is denoted by $\cl{D}_\lis(\Bun_G,\Qla)_\cusp$. Its objects are called cuspidal. 
\end{defi}

Cuspidal sheaves are supported on $B(G)_\bas$ since every stalk at a non-basic $\sigma$-conjugacy class is written as a direct summand of a geometric constant term and thus vanishes. Thus, we can regard cuspidal sheaves as families of representations of $G_b(E)$ over $b\in B(G)_\bas$. We are going to study what kind of representations are cuspidal in this sense. 

For every $A\in \cl{D}_\lis(\Bun_G,\Qla)_{[G]}$, we see that $\Hom(\Eis_{\ov{P}}(\CT_P(A)),A)=0$ for every proper parabolic subgroup $P$ of $G$ since the projection of $\Eis_{\ov{P}}(\CT_P(A))$ to $\cl{D}_\lis(\Bun_G,\Qla)_{[G]}$ is zero. In particular, we have $\CT_P(A)=0$ for every $P$, so $A$ is cuspidal. Let us see that the essential surjectivity of $\Eis_{P,[M]}$ for all $P$ implies the converse direction. Suppose $A$ is cuspidal and lies in $\cl{D}_\lis(\Bun_G,\Qla)_{[M]}$ for a Levi subgroup $M$ of $G$. The hypothesis on the essential surjectivity implies that there is an object $B\in \cl{D}_\lis(\Bun_M,\Qla)_{[M]}$ such that $\Hom(\Eis_P(B),A)\neq 0$ for some parabolic subgroup $P$ with a Levi subgroup $M$. The second adjointness implies $\Hom(B,\CT_\ov{P}(A))\neq 0$, but it contradicts the cuspidality of $A$. Thus, we arrive at the following conjecture. 

\begin{conj}\label{conj:cusp}
	The inclusion $\cl{D}_\lis(\Bun_G,\Qla)_{[G]} \subset \cl{D}_\lis(\Bun_G,\Qla)_\cusp$ is an equality. In particular, for any smooth irreducible $\Qla$-representation $\pi$ of $G_b(E)$ with $b\in B(G)_\bas$, its Fargues--Scholze parameter $\varphi^\FS_\pi$ is supercuspidal if and only if $\CT_P(i^b_{!}\pi)=0$ for every proper parabolic subgroup $P$ of $G$. 
\end{conj}

What is remarkable here is that in the usual representation theory, $\varphi^\FS_\pi$ is not necessarily supercuspidal even if $\pi$ itself is supercuspidal. This is easier to see when $G_b$ is anisotropic, but it is also the case when $G_b=\GSp_4$. This difference expresses the fact that geometric Eisenstein series can generate those representations which cannot be generated by representation-theoretic parabolic inductions. To grasp the behavior of geometric Eisenstein series, we will study two concrete examples $G=\GL_n$ and $G=\GSp_4$ in the following.

Through these examples, we observe that those representations generated by geometric Eisenstein series are mainly  generated by a combination of representation-theoretic parabolic inductions and Hecke operators. The relation between the above conjecture and the characterization of supercuspidal $L$-packets in \cite[\S 3.5]{DR09}  can be explained through this observation. 
For every smooth irreducible $\Qla$-representation $\pi$, we can expect from the categorical local Langlands correspondence that each member of the $L$-packet containing $\pi$ appears as a subquotient of $T_V\pi$ for some $V\in \Rep(\LG)$. The expectation in \cite[\S 3.5]{DR09} is that the $L$-packet containing $\pi$ consists only of supercuspidal representations if and only if the $L$-parameter of $\pi$ is supercuspidal. Thus, if $\varphi_\pi^\FS$ is not supercuspidal, then it is expected that a non-supercuspidal representation in the $L$-packet containing $\pi$ appears as a subquotient of $T_V\pi$ and it implies that a geometric constant term of $i^b_!\pi$ is nonzero. 

\subsubsection{The case $G=\mrm{GL}_n$}\label{sssc:GLn}

This case enjoys a particular property that for every supercuspidal irreducible $\Qla$-representation $\pi$ of $\GL_n(E)$, the $L$-packet containing $\pi$ is a singleton and its $L$-parameter is supercuspidal. In particular, the above conjecture does not claim much about smooth representations of $\GL_n(E)$. We deduce the claim for extended pure inner forms $G_b(E)$ from the following conjectural property of the compatibility of geometric Eisenstein series with Hecke operators. 

Let $P$ be a parabolic subgroup of $G$ with a Levi subgroup $M$ and take $V\in \Rep (\LG)$. Let $U$ be the unipotent radical of $P$. 

\begin{conj}\textup{(\cite[Conjecture 1.5.2]{Han24})}\label{conj:compat}
	Choose a finite filtration $0=V_0\subset V_1\subset \cdots \subset V_m=V\vert_{\LP}$ such that $\LU$-action on each graded piece $W_i=V_i/V_{i-1}$ is trivial, i.e. such that each $W_i$ is naturally inflated from $\Rep(\LM)$. Then, $T_V\Eis_P(-)$ admits a corresponding finite filtration with graded pieces $\Eis_P(T_{W_i}-)$. 
\end{conj}

When the derived subgroup of $G$ is simply connected and $P$ is a Borel subgroup, this is proved in \cite{Ham23} with torsion coefficients along the same line as in the classical geometric case in \cite{BG02}. As noted in \cite{Han24}, it is expected that the general case can be handled in the same way with a suitable sheaf theory. 

\Cref{conj:compat} implies that the essential images of geometric Eisenstein series are stable under Hecke operators. Thus, for every $A\in \cl{D}_\lis(\Bun_G,\Qla)_\cusp$ and $B\in \cl{D}_\lis(\Bun_M,\Qla)$, we have 
$\Hom(B,\CT_{\ov{P}}(T_V A))=\Hom(T_{V^\vee}\Eis_P(B),A) =0$
since $T_V$ is left adjoint to $T_{V^\vee}$. Thus, we see that $\cl{D}_\lis(\Bun_G,\Qla)_\cusp$ is also stable under Hecke operators. 

Now, suppose that we have a nonzero cuspidal sheaf $A\in \cl{D}_\lis(\Bun_G,\Qla)_\cusp$ which lies in $\cl{D}_\lis(\Bun_G,\Qla)_{[M]}$ for some proper Levi subgroup $M$. We may assume that $A$ is supported on a basic $\sigma$-conjugacy class $b\in B(G)_\bas$. Take $V\in \Rep(\LG)$ so that its central character corresponds to $b$ under the Kottwitz isomorphism $X^*(Z(\widehat{G})^\Gamma)\cong B(G)_\bas$. By replacing $A$ with $T_{V^\vee} A$, we may suppose that $A$ is supported at $b=1$. Since the Jacquet functor is exact, every cohomology group of $A$ is a smooth representation of $\GL_n(E)$ whose Jacquet modules are all zero. In particular, every subquotient of a cohomology group of $A$ is supercuspidal, but it also lies in $\cl{D}_\lis(\Bun_G,\Qla)_{[M]}$. This is a contradiction. In particular, we show the following. 

\begin{prop}
	When $G=\GL_n$, \Cref{conj:cusp} follows from \Cref{conj:compat}. 
\end{prop}

\subsubsection{The case $G=\GSp_4$ over $E=\bb{Q}_p$}\label{sssc:GSp4}
Let $\varphi\colon W_{\Qp}\to \GSp_4(\Qla)$ be an $L$-parameter for $\GSp_4$. In this case, there are two kinds of non-supercuspidal $L$-parameters $\varphi$ whose $L$-packets $\Pi_\varphi$ contain supercuspidal representations. They are called local Saito--Kurokawa and local Howe--Piatetski-Shapiro $L$-parameters. First, we study the case of local Saito--Kurokawa type following the work announced in \cite{Mie21} and the explanation in \cite[Example 3.3.2]{Han24}. 

Let $\pi$ be an irreducible smooth supercuspidal $\Qla$-representation of $\GSp_4(\Qp)$ whose $L$-parameter $\varphi$ is of local Saito--Kurokawa type. 
Let $D$ be the quaternion division algebra over $\Qp$ and consider the inner form $J=\mrm{GU}_2(D)$ of $G=\GSp_4(\Qp)$. According to the local Langlands correspondence for $G$ in \cite{GT11} and for $J$ in \cite{GT14}, both of the $L$-packets $\Pi_\varphi^G$ and $\Pi_\varphi^J$ of $\varphi$ for $G$ and $J$, respectively, consist of one supercuspidal representation and one non-supercuspidal discrete representation. Let $\rho$ be the non-supercuspidal discrete representaion of $J(\Qp)$ in $\Pi_\varphi^J$. The work of Ito--Mieda announced in \cite{Mie21} calculates the supercuspidal part of the $\rho$-isotypic component of the $\ell$-adic cohomology of the basic Rapoport-Zink space for $\GSp_4$ and finds the nontrivial contribution of $\pi$ there. We will review this result more precisely as follows. 

The standard representation of $\GSp_4$ is denoted by $\std\colon \GSp_4 \hookrightarrow \GL_4$. Let $b\in B(G)_\bas$ be the basic $\sigma$-conjugacy class corresponding to the central character of $\std$ and let $\mu$ be the highest weight of $\std$. The local Shimura datum $(G,b,\mu)$ and a compact open subgroup $K\subset G(\Qp)$ define a local Shimura variety $\Sht_{(G,b,\mu,K)}$ at the level $K$, also known as a Rapoport-Zink space. Consider the $\ell$-adic \'{e}tale cohomology of the tower $\{\Sht_{(G,b,\mu,K)}\}_K$
\[R\Gamma_c(G,b,\mu) = \colim_K R\Gamma_c(\Sht_{(G,b,\mu,K)},\Qla[d](\tfrac{d}{2}))\]
with $d=3$ the dimension of a local Shimura variety $\Sht_{(G,b,\mu,K)}$. 
It is a representation of $G(\Qp)\times J(\Qp) \times W_\Qp$ and its $\rho$-isotypic component
$\RHom_{J(\Qp)}(R\Gamma_c(G,b,\mu),\rho)^\sm$
is denoted by $R\Gamma(G,b,\mu)[\rho]$. Here, $(-)^\sm$ denotes the operation restricting to the set of smooth $G(\Qp)$-vectors, i.e. $(-)^\sm=\colim_K (-)^K$. 
The $\rho$-isotypic component $R\Gamma(G,b,\mu)[\rho]$ is a representation of $G(\Qp)\times W_\Qp$ and as in \cite[Section IX.3]{FS21}, it can be described as $R\Gamma(G,b,\mu)[\rho]=i^{1*}T_{\std^\vee}i^b_{*} \rho$. 
Now, we explain the computation of $R\Gamma(G,b,\mu)[\rho]$ by Ito--Mieda. Here, we do not consider the action of $W_\Qp$ which is treated in their computations. 

\begin{prop}\textup{(\cite[Theorem 3.1]{Mie21})}\label{prop:ItoMieda}
	The supercuspidal part of $R\Gamma(G,b,\mu)[\rho]$ is concentrated on degree $0$ and isomorphic to $\pi^{\oplus 2}$ as a representation of $G(\Qp)$. 
\end{prop}

Assuming \Cref{conj:compat} and this computation, we will show that $i^1_!\pi$ is not cuspidal. For this, it is enough to show that $\pi$ lies in the subcategory $\cl{D}_\lis(\Bun_G,\Qla)^\Eis$ of $\cl{D}_\lis(\Bun_G,\Qla)$ generated by the essential images of $\Eis_P$ for all proper parabolic subgroups $P$. Since $\cl{D}_\lis(\Bun_G,\Qla)^\Eis$ contains all lisse-\'{e}tale sheaves supported on non-basic $\sigma$-conjugacy classes, it is enough to keep track of the restriction to the semistable locus. 
First, since $\rho$ is non-supercuspidal, it has a resolution by parabolically induced representations, so we see that $i^b_{!}\rho$ lies in $\cl{D}_\lis(\Bun_G,\Qla)^\Eis$ and so does $i^b_{*}\rho$. Then, we see from \Cref{conj:compat} that $\cl{D}_\lis(\Bun_G,\Qla)^\Eis$ is stable under Hecke operators, so $R\Gamma(G,b,\mu)[\rho]=i^{1*}T_{\std^\vee}i^b_{*} \rho$ also lies in $\cl{D}_\lis(\Bun_G,\Qla)^\Eis$. Ito--Mieda's computation implies that it contains $\pi$ as a direct summand, so it follows that $\pi$ lies in $\cl{D}_\lis(\Bun_G,\Qla)^\Eis$. 
Summarizing, we show assuming \Cref{conj:compat} and \Cref{prop:ItoMieda} that for any supercuspidal representation $\pi$ of $\GSp_4(\Qp)$ of local Saito--Kurokawa type, $i^1_{!}\pi$ lies in $\cl{D}_\lis(\Bun_G,\Qla)^\Eis$ and in particular, it is not cuspidal. 

Now, we turn to the general case. Let $\pi$ be an irreducible smooth supercuspidal $\Qla$-representation of $\GSp_4(\Qp)$ with a non-supercuspidal $L$-parameter $\varphi$. Let $\rho$ be a non-supercuspidal representation in the $L$-packet $\Pi_\varphi^J$ (which is not unique if $\varphi$ is of Howe--Piatetski-Shapiro type). In general, exact computations such as \Cref{prop:ItoMieda} are not available in the literature. However, in the above argument, we only need the fact that $\pi$ appears as a direct summand of $R\Gamma(G,b,\mu)[\rho]$. Let $Z$ be the center of $G$. Since $J$ is an inner form of $G$, $Z$ can be identified with the center of $J$. Since the diagonal action of $Z(\Qp)$ on $R\Gamma_c(G,b,\mu)$ is trivial and $\rho$ has a central character $\omega_\rho$, the action of $Z(\Qp)$ on $R\Gamma(G,b,\mu)[\rho]$ is scalar by the character $\omega_\rho$. Thus, the supercuspidal part of $R\Gamma(G,b,\mu)[\rho]$ is a direct sum of supercuspidal representations of $G(\Qp)$ with central characters $\omega_\rho$. It is enough to show that $\pi$ has a nonzero coefficient in the Euler characteristic of $R\Gamma(G,b,\mu)[\rho]$, denoted by $\Mant_{b,\mu}(\rho)$. 
Then, we use the formula of $\Mant_{b,\mu}(\rho)$ obtained in \cite{HKW22}. 
Let $\delta_{\pi,\rho}$ be the algebraic representation of the $S$-group $S_\varphi$ attached to $\pi$ relative to $\rho$. Since the error term in the formula of \cite{HKW22} has no supercuspidal part (see \cite[Appendix C]{HKW22}), the coefficient of $\pi$ in $\Mant_{b,\mu}(\rho)$ is $\dim \Hom_{S_\varphi}(\delta_{\pi,\rho},\std)$ which is seen to be nonzero. Thus, we get the following. 
\begin{prop}
	If we assume \Cref{conj:compat}, it follows that for any supercuspidal representation $\pi$ of $\GSp_4(\Qp)$, $i^1_{!}\pi$ lies in $\cl{D}_\lis(\Bun_G,\Qla)^\Eis$ if and only if $\varphi_\pi^\FS$ is not supercuspidal, and in particular, the latter claim of \Cref{conj:cusp} holds for $G=\GSp_4$ over $\Qp$. 
\end{prop}
\begin{proof}
	It is enough to show the second claim for inner twists $G_b(\Qp)$. Suppose that $i_!^b \pi$ is cuspidal. Then, for a suitable $V\in \Rep(\LG)$, $T_Vi_!^b\pi$ is supported at $[1]\in B(G)$ and its cohomology groups are of finite length (see \cite[Section IX.3]{FS21}). Since $T_Vi_!^b\pi$ is also cuspidal with a central character $\omega_\pi$ of $\pi$, it has no non-supercuspidal part and it is decomposed into a direct sum of supercuspidal representations. Each direct summand should be cuspidal with Fargues--Scholze parameter $\varphi^\FS_\pi$, so $\varphi^\FS_\pi$ should be supercuspidal. 
	Note that $T_Vi_!^b\pi$ is nonzero since $T_{V^\vee} T_Vi_!^b\pi$ contains $i_!^b\pi$ as a direct summand. 
\end{proof}

The above argument is similar to the one used to show the compatibility of Fargues--Scholze parameters with the classical local Langlands correspondence in \cite{Ham22b} and \cite{BHN22}. In particular, the same argument works for unramified unitary and unitary similitude groups in an odd number of variables. 

\bibliography{reference}
\bibliographystyle{alpha}
\end{document}